\documentclass[13pt, reqno]{amsart}
\usepackage{dsfont}
\usepackage{amsmath}
\usepackage{amscd}
\usepackage{amsthm}
\usepackage{amssymb} \usepackage{latexsym}
\usepackage{eufrak}
\usepackage{euscript}
\usepackage{epsfig}
\usepackage{graphics}
\usepackage{array}
\usepackage{enumerate}
\usepackage{color}
\usepackage{wasysym}
\usepackage{pdfsync}
\usepackage{stmaryrd}
\usepackage{url}
\usepackage[font=small, labelformat=simple]{caption}

\newcommand{\bel}[1]{\begin{equation*}\label{#1}}
	
	\newcommand{\be}{\begin{equation}}

		\newcommand{\ba}{\begin{eqnarray}}
			\newcommand{\ea}{\end{eqnarray}}

		\newcommand{\qe}{\end{equation}}

	\newcommand{\Z}{{\mathbb Z}}

	\newcommand{\supp}{{\mathrm{supp}}}

	\newcommand{\eg}{\begin{example}}
		\newcommand{\egd}{\end{example}}
	\newcommand{\tm}{\begin{thm}}
		\newcommand{\tmd}{\end{thm}}
	\newcommand{\co}{\begin{coro}}
		\newcommand{\cod}{\end{coro}}
	\newcommand{\enu}{\begin{enumerate}}
		\newcommand{\enud}{\end{enumerate}}
	\newcommand{\rmk}{\begin{rem}}
		\newcommand{\rmkd}{\end{rem}}
	
	\theoremstyle{theorem}
	\newtheorem{thm}{Theorem}[section]

	\theoremstyle{example}
	\newtheorem{example}[thm]{Example}
	\newtheorem{coro}[thm]{Corollary}
	\theoremstyle{lemma}
	\newtheorem{lemma}[thm]{Lemma}
	\theoremstyle{definition}
	\newtheorem{defi}[thm]{Definition}
	\theoremstyle{proof}
	
	\theoremstyle{remark}
	\newtheorem{rem}[thm]{Remark}
	\theoremstyle{remark}

	\UseRawInputEncoding

\begin{document}

		\title[scattering theory of the nonlinear wave equations on lattice graphs]{ scattering theory of the nonlinear wave equations on lattices}

		\author{Jiajun Wang}
		\address{Jiajun Wang: School of Mathematical Sciences,
			Fudan University, Shanghai 200433, China.}
		\email{21300180146@m.fudan.edu.cn}
		
		\begin{abstract}
			In this paper, I will summarize the uniform decay estimates of the discrete wave equations (DW) established by the oscillatory integral theory in \cite{1,2,3}, and combine the abstract framework of the scattering theory of the dispersive equations established in \cite{4} to finally establish the scattering theory of the discrete nonlinear wave equations (DNLW), including the existence of the wave operators and the asymptotic completeness. In addition, I will establish the scattering theory again by the Strichartz estimate.
		\end{abstract}

		\maketitle
		\numberwithin{equation}{section}
		\section{Introduction}
		\subsection{Background}
		Discrete equations have gathered a lot of interests both from mathematics and physics. 
		
		At the physical level, the discrete equations are models of many physical phenomena. For example, the discrete wave equation can describe the vibration of atoms in crystalline semiconductors \cite{12}, the discrete Schr\"{o}dinger equation can describe the non-exponential energy relaxation \cite{13} in solids, and the discrete heat equation can also reveal the law of interlayer thermal resistance and its dispersion properties \cite{14}. 
		
		At the mathematical level, the discrete equations also serve as discretizations of some conventional equations in continuous space, thereby providing a direct theoretical foundation for numerical approximation. For example, the continuum limit theory of PDEs is to establish that, after discretization, the solution of the discrete equations can approximate the solution of PDEs in the original continuous space. Concretely, for the fractional nonlinear Schr\"{o}dinger equations, \cite{15} established the continuum limit theory for 1 dimensional case, in a weak convergence sense; \cite{16} introduced the uniform Strichartz estimate to improve the previous result into a strong convergence sense; \cite{17} generalized the continuum limit theory into 2 dimension, with much more difficult tools in the oscillatory integral theory; the author continued to extend the continuum limit theory into 3 dimension in \cite{8}.
		
		This paper is to study the scattering theory of DNLW. Compared with my research on the scattering theory of the discrete nonlinear Schr\"{o}dinger equations in \cite {18}, the study of DNLW is more difficult. Specifically, the propagator of DW will produce singularity, and we can't use the separation of variables to analyze the corresponding oscillatory integrals. In fact, using basic knowledge of the oscillatory integrals, we can get the optimal uniform decay rate of the solution to the discrete Schr\"{o}dinger equation. At present, the uniform decay estimate of the solution to DW is only carried out to 5D \cite{1,2,3}. Therefore, the use of the Strichartz estimate in DNLW will be more complicated and  cumbersome. This directly leads to greater difficulties in establishing the scattering results of DNLW. 
		\vspace{10pt}
		\subsection{Basic terminology in discrete equations} 
		To state our major results, we first introduce DNLW.
			\begin{equation}\label{DNLW}
			\left\{
			\begin{aligned}
				& \partial_{t}^{2} u(x,t)-\Delta u(x,t) =\mu |u|^{p-1}u,  \\
				& u(x,0) = f(x), \partial_{t}u(x,0)=g(x),\quad (x,t)\in \mathbb{Z}^d\times \mathbb{R},
			\end{aligned}
			\right.
		\end{equation}
		where $u:\mathbb{Z}^d\times \mathbb{R}\to \mathbb{C}$, $f,g:\mathbb{Z}^{d}\to \mathbb{C}$, $p>1$, $\mu=\pm1$. The case $\mu=1$ is called the focusing case and $\mu=-1$ is called the defocusing case, and the Laplacian is referred to the following operator.
		
		Suppose $G=(V,E)$ is a locally finite simple graph, $V$ is vertex set, $E$ is edge set. For $x,y\in V$, we denote $x\sim y$, if $x$,$y$ are adjacent, i.e. $(x,y)\in E$.  Then the Laplacian $\Delta$ on graph is defined as 
		\begin{equation*}
			\Delta f(x):=\sum_{y\in V, y\sim x}\Big(f(y)-f(x)\Big), \quad x\in V,
		\end{equation*}
		where $f:V\to \mathbb{C}$.
		
		Next, we briefly recall the $L^{p}$ space on $G=(V,E)$, which we usually denote as $\ell^{p}(G)$. For $f:V\to \mathbb{C}$, we define the $\ell^{p}(G)$-norm as
		\begin{equation*}
			\|f\|_{\ell^{p}(G)}:=\left(\sum_{v\in G}|f(v)|^{p}\right)^{\frac{1}{p}},
		\end{equation*}
		with the corresponding $\ell^{p}(G)$-space:
		\begin{equation*}
			\ell^{p}(G):=\left\lbrace f:G\to\mathbb{C}\big|\|f\|_{\ell^{p}(G)}<\infty\right\rbrace.
		\end{equation*}
		\begin{rem}
			In the following discussion, we only consider the case of lattices, i.e. $G=\mathbb{Z}^{d}$. Then a trivial observation is $\ell^{p}\subseteq\ell^{q}$, $0<p<q\le\infty$. This simple property will, in some cases, simplify our proofs. 
		\end{rem}
		
		\vspace{10pt}
		\subsection{Abstraction of the nonlinear dispersive equation}
		Before formally stating our major results, we shall introduce some concepts related to the  scattering theory, which can be referred to \cite{4}. We first consider the following nonlinear dispersive equation on a
		Hilbert space $X$.
		\begin{equation}\label{abstract}
			\frac{du}{dt}=iH_{0}u+Pu,
		\end{equation}
		where $P$ is a densely defined nonlinear operator, $H_{0}$ is a densely defined linear operator, and $U_{0}(t):=\exp(itH_{0})$ is the corresponding unitary operator on $X$. 
		\begin{defi}
			We call $u$ a strong solution to the equation (\ref{abstract}), if $u\in C(I;X)$ and satisfies 
			\begin{equation}\label{initial}
				u(t)=U_{0}(t)u(0)+\int_{0}^{t}U_{0}(t-s)Pu(s)ds, \quad \forall t\in I.
			\end{equation}
		\end{defi}
		The classic scattering theory includes the existence of the wave operator and the asymptotic completeness, which are defined as follows.
		\begin{defi}
			If $f_{\pm}\in X \;$, and there exists a strong solution $u$ of the equation (\ref{abstract}), such that
			\begin{equation}\label{wave operator}
				\|u(t)-u_{\pm}(t)\|_{X}\to0, \quad t\to \pm\infty,
			\end{equation}
			where $u_{\pm}(t)=U_{0}(t)f_{\pm}$. Then we define the wave operator $W_{\pm}(f_{\pm})=u(0)$.
		\end{defi}
		\begin{rem}
			We also define the scattering operator as $S:f_{-}\mapsto f_{+}$.
		\end{rem}
		\begin{defi}
			A solution $u$ is said to be asymptotic complete if it is a strong solution of the equation (\ref{abstract}) and there exist $f_{\pm}\in X$ satisfying (\ref{wave operator}).
		\end{defi}
		For convenience, some frequently encountered integral equations are denoted as follows. 
		\begin{itemize}
			\item ($\ast_{s},f$):\begin{equation*}
				u(t)=U_{0}(t)f+\int_{s}^{t}U_{0}(t-s)Pu(s)ds 
			\end{equation*}
			\item ($\ast_{\pm\infty},f_{\pm}$):\begin{equation*}
				u(t)=U_{0}(t)f_{\pm}+\int_{\pm\infty}^{t}U_{0}(t-s)Pu(s)ds 
			\end{equation*}
		\end{itemize}
		\begin{rem}
			Notice that the solution of the integral equation ($\ast_{s},f$), is the strong solution of the equation (\ref{abstract}), with the initial data $U_{0}(s)f$. And, informally, the integral equation ($\ast_{\pm\infty},f_{\pm}$) is equivalent to (\ref{wave operator}). Specifically, as $U_{0}(t)$ is a unitary operator on $X$, (\ref{wave operator}) is equivalent to 
			\begin{equation*}
				\|U_{0}(-t)u(t)-f_{\pm}\|_{X}\to 0, t\to \pm\infty.
			\end{equation*}
			From ($\ast_{\pm\infty},f_{\pm}$), it is equivalent to
			\begin{equation*}
				\left\|\int_{\pm\infty}^{t}U_{0}(-s)Pu(s)ds\right\|_{X}\to 0, t\to \pm\infty.
			\end{equation*}
			Thus, the scattering requires some integrability conditions on the solution $u$.
		\end{rem}
		\begin{rem}\label{transform}
			To fit DNLW (\ref{DNLW}) into the framework of (\ref{abstract}), we can use the following notations:
			\begin{equation*}
				\widetilde{u}(t):=
				\begin{bmatrix}
					u(t)  \\
					
					\partial_{t}u(t) \\
				\end{bmatrix},\quad 
				P: \begin{bmatrix}
					u_{1} \\
					
					u_{2} \\
				\end{bmatrix}\mapsto
				\begin{bmatrix}
					0 \\
					
					\mu|u_{1}|^{p-1}u_{1} \\
				\end{bmatrix},\quad H_{0}:=
				\begin{bmatrix}
					0 & -i\\
					
					-i\Delta& 0\\
				\end{bmatrix}.
			\end{equation*}
			Then we see that DNLW (\ref{DNLW}) is equivalent to 
			\begin{equation*}
				\frac{d\widetilde{u}}{dt}=iH_{0}\widetilde{u}+P\widetilde{u}.
			\end{equation*}
		And if there is no further specification, we use $\widetilde{u}$ to denote $(u,\partial_{t}u)$.
		\end{rem}
		\vspace{10pt}
		\subsection{Major theorems and results}
		To simplify the statements of our major results, we have the following notations. 
	
	    We denote $\beta_{2}:=\left(\frac{3}{4}\right)^{-}, \beta_{3}:=\frac{7}{6},  \beta_{4}:=\left(\frac{3}{2}\right)^{-}, \beta_{5}:=\frac{11}{6}$, where $a^{-}$ represents any number less than $a$. And $p_{d}$ is the bigger root of the equation $\beta_{d}=\frac{p+1}{p(p-3)}$, $d=2,3,4,5$.
		
		In order to correspond with the framework in \cite{4}, we take $X_{3}:=\ell^{p+1}(\mathbb{Z}^{d})\times\ell^{p+1}(\mathbb{Z}^{d})$, $Y:=\ell^{1+\frac{1}{p}}(\mathbb{Z}^{d})\times\ell^{1+\frac{1}{p}}(\mathbb{Z}^{d})$, $X:=\dot{H}^{1}(\mathbb{Z}^{d})\times \ell^{2}(\mathbb{Z}^{d})$ (the definition of $\dot{H}^{1}(\mathbb{Z}^{d})$ can be seen in Section 4).
		\begin{thm}\label{main}
			For $2\le d\le 5$, $p>p_{d}$, there exists $\delta>0$, such that if $F_{-}\in X\cap Y$, $\|F_{-}\|_{X}+\|F_{-}\|_{Y}<\delta$, then we have a unique $\widetilde{u}\in C(\mathbb{R};X)$ satisfying ($\ast_{-\infty},F_{-}$) and 
			\begin{equation}\label{rr}
				\sup_{t\in\mathbb{R}}(1+|t|)^{\beta_{d}(1-\frac{4}{p+1})}\|\widetilde{u}(t)\|_{X_{3}}<\infty.
			\end{equation}
			Most importantly, we have
			\begin{equation}\label{qq}
				\|\widetilde{u}(t)-U_{0}(t)F_{-}\|_{X}\to 0, \quad t\to-\infty.
			\end{equation}
			Furthermore, there exists a unique $F_{+}\in X$, such that $\widetilde{u}$ satisfies ($\ast_{+\infty},F_{+}$) and
			\begin{equation*}
				\|\widetilde{u}(t)-U_{0}(t)F_{+}\|_{X}\to 0, \quad t\to+\infty,
			\end{equation*}
			with the following conservation of energy. 
			\begin{equation}\label{conservation}
				\frac{1}{2}\|\widetilde{u}(t)\|_{X}^{2}-\frac{\mu}{p+1}\|u(t)\|_{\ell^{p+1}(\mathbb{Z}^{d})}^{p+1}=\frac{1}{2}\|F_{-}\|_{X}^{2}=\frac{1}{2}\|F_{+}\|_{X}^{2}, \quad \forall t\in \mathbb{R}.
			\end{equation}
		\end{thm}
		\begin{thm}\label{main2}
			For $2\le d\le 5$, $p>p_{d}$, if $F_{-}\in X\cap Y$, then there exsit $T>-\infty$ and a unique $\widetilde{u}\in C((-\infty,T];X\cap X_{3})$, satisfying ($\ast_{-\infty},F_{-}$) and
			\begin{equation*}
				\sup_{-\infty<t\le T}(1+|t|)^{\beta_{d}(1-\frac{4}{p+1})}\|\widetilde{u}(t)\|_{X_{3}}<\infty.
			\end{equation*}
			Similarly, we also have (\ref{qq}) and the following energy conservation. 
			\begin{equation*}
				\frac{1}{2}\|\widetilde{u}(t)\|_{X}^{2}-\frac{\mu}{p+1}\|u(t)\|_{\ell^{p+1}(\mathbb{Z}^{d})}^{p+1}=\frac{1}{2}\|F_{-}\|_{X}^{2}, \quad \forall t\in (-\infty,T].
			\end{equation*}
		\end{thm}
		\begin{thm}\label{main3}
			For $2\le d\le 5$, $p>p_{d}$, there exists $\delta>0$, such that if $F_{0}\in X\cap Y$, and $\|F_{0}\|_{X}+\|F_{0}\|_{Y}<\delta$, then we have a unique strong solution $\widetilde{u}\in C(\mathbb{R};X)$ of the equation (\ref{DNLW}), satisfying (\ref{rr}). Moreover, there exists $F_{\pm}\in X$, satisfying all assertions of Theorem \ref{main}.
		\end{thm}
		The above results are all based on the framework in \cite{4}. Although this framework can be applied to other cases, it may not be a refined tool to study DNLW. For example, this framework doesn't distinguish the defocusing case and the focusing case, whose behaviors are extremely different. In the following content, we will use the Strichartz estimate to reconstruct the scattering theory of DNLW.
		
		To use the Strichartz estimate, we first introduce the admissible pair of the Strichartz estimate and the Strichartz norm.
		\begin{defi}
			($q,r$) is called $\sigma$-admissible, if $q,r\ge 2$, and 
			\begin{equation*}
				(q,r,\sigma)\ne (2,\infty,1), \quad \frac{1}{q}+\frac{\sigma}{r}\le \frac{\sigma}{2}.
			\end{equation*}
		\end{defi}
		\begin{defi}\label{SS}
			For the time interval $I\subseteq\mathbb{R}$, we define the Strichartz space $S^{0}(I\times\mathbb{Z}^{d})$, $S^{1}(I\times \mathbb{Z}^{d})$ as the closures of the Schwarz space under the following norms.
			\begin{equation*}
				\|u\|_{S^{0}(I\times\mathbb{Z}^{d})}:=\sup_{(q,r)\; \text{is}\; \beta_{d}-\text{admissible}}\|u\|_{L_{t}^{q}\ell^{r}_{x}(I\times\mathbb{Z}^{d})}.
			\end{equation*}
			\begin{equation*}
				\|u\|_{S^{1}(I\times\mathbb{Z}^{d})}:=\sup_{(q,r)\; \text{is}\; \beta_{d}-\text{admissible}}\|u\|_{L_{t}^{q}\dot{W}_{x}^{r}(I\times\mathbb{Z}^{d})}.
			\end{equation*}
			For $\widetilde{u}=(u_{1},u_{2})$, we similarly define $S(I\times \mathbb{Z}^{d})$ as the closure of the Schwarz space under the following norm.
			\begin{equation*}
				\|\widetilde{u}\|_{S(I\times \mathbb{Z}^{d})}:=\|u_{1}\|_{S^{1}(I\times\mathbb{Z}^{d})}+\|u_{2}\|_{S^{0}(I\times\mathbb{Z}^{d})}.
			\end{equation*}
			To fully use the Strichartz estimate, we also define $N^{0}(I\times\mathbb{Z}^{d})$, $N^{1}(I\times\mathbb{Z}^{d})$,and $N(I\times\mathbb{Z}^{d})$, as the dual space of $S^{0}(I\times\mathbb{Z}^{d})$, $S^{1}(I\times\mathbb{Z}^{d})$, and $S(I\times\mathbb{Z}^{d})$, respectively.
		\end{defi}
		\begin{rem}
			For $F=(F_{1},F_{2})$, we abbreviate $\|\cdot\|_{A\times A}$ as $\|\cdot\|_{A}$. For example, 
			\begin{equation*}
				\|F\|_{L_{t}^{q}\ell^{r}_{x}(I\times\mathbb{Z}^{d})}:=\|F_{1}\|_{L_{t}^{q}\ell^{r}_{x}(I\times\mathbb{Z}^{d})}+\|F_{2}\|_{L_{t}^{q}\ell^{r}_{x}(I\times\mathbb{Z}^{d})},
			\end{equation*}
			\begin{equation*}
				\|F\|_{S^{i}(I\times \mathbb{Z}^{d})}:=\|F_{1}\|_{S^{i}(I\times \mathbb{Z}^{d})}+\|F_{2}\|_{S^{i}(I\times \mathbb{Z}^{d})}, \quad i=0,1.
			\end{equation*}
		\end{rem}
		
		Next, we can use the Strichartz estimate to improve Theorem \ref{main} from different perspectives. For instance, in the defocusing case, Theorem \ref{S1} and Theorem \ref{S1.5} remove the smallness condition on the asymptotic state $F_{-}$.
		
		\begin{thm}\label{S1}
			For $2\le d\le 5$, $p\ge 1+\frac{2}{\beta_{d}}+\frac{2}{d}$, $\mu=-1$, if $F_{-}=(f,g)\in \ell^{2}(\mathbb{Z}^{d})\times\ell^{\frac{2d}{d+2}}(\mathbb{Z}^{d})$, then there exists a unique $\widetilde{u}\in C(\mathbb{R};X)$ satisfying $(\ast_{-\infty},F_{-})$ and
			\begin{equation*}
				\|\widetilde{u}(t)-U_{0}(t)F_{-}\|_{X}\to 0, \quad t\to-\infty.
			\end{equation*}
			Similarly, for $F_{+}\in \ell^{2}(\mathbb{Z}^{d})\times\ell^{\frac{2d}{d+2}}(\mathbb{Z}^{d})$, there exists $\widetilde{u}$, not necessarily the same `` $\widetilde{u}$" for $F_{-}$, satisfying $(\ast_{+\infty},F_{+})$, with 
			\begin{equation*}
				\|\widetilde{u}(t)-U_{0}(t)F_{+}\|_{X}\to 0, \quad t\to+\infty,
			\end{equation*}
			and the energy conservation $(\ref{conservation})$. Furthermore, there also exists $\delta>0$, such that if $\|f\|_{\ell^{2}(\mathbb{Z}^{d})}+\|g\|_{\ell^{\frac{2d}{d+2}}(\mathbb{Z}^{d})}\le \delta$, then $\widetilde{u}\in S^{0}(\mathbb{R}\times\mathbb{Z}^{d})\times S^{0}(\mathbb{R}\times\mathbb{Z}^{d})$.
		\end{thm}
		\begin{thm}\label{S1.5}
			For $3\le d\le 5$, $p\ge \frac{d(\beta_{d}+2)}{\beta_{d}(d-2)}$, $\mu=-1$, if $F_{-}\in X:= \dot{H}^{1}(\mathbb{Z}^{d})\times\ell^{2}(\mathbb{Z}^{d})$, then the assertions in Theorem \ref{S1} can also be satisfied. Moreover, there exists $\delta>0$, such that if $\|F_{-}\|_{X}\le  \delta$, then $\widetilde{u}\in S(\mathbb{R}\times\mathbb{Z}^{d})$.
		\end{thm}
		\begin{rem}
			Compared with Theorem \ref{main}, Theorem \ref{S1} broadens the range of the nonlinear power $p$, when $d\ge3$. In fact, simple calculation shows that $p_{d}>1+\frac{2}{\beta_{d}}+\frac{2}{d}$, when $d\ge3$. However, they all impose extra hypothesis on the asymptotic state $F_{-}$, other than $F_{-}\in X$. For example, Theorem \ref{main} requires $F_{-}\in Y=\ell^{1+\frac{1}{p}}(\mathbb{Z}^{d})\times\ell^{1+\frac{1}{p}}(\mathbb{Z}^{d})$; Theorem \ref{S1} requires $F_{-}\in \ell^{2}(\mathbb{Z}^{d})\times\ell^{\frac{2d}{d+2}}(\mathbb{Z}^{d})$. Theorem \ref{S1.5}, instead, trades the range of nonlinear power $p$ for removing the extra hypothesis on $F_{-}$.
		\end{rem}
		Next, we can also use the Strichartz estimate to reconstruct the asymptotic completeness.
		\begin{thm}\label{S2}
			For $2\le d\le 5$, $p\ge 1+\frac{2}{\beta_{d}}+\frac{2}{d}$, there exists $\delta>0$, such that if $F_{0}=(f,g)\in \ell^{2}(\mathbb{Z}^{d})\times\ell^{\frac{2d}{d+2}}(\mathbb{Z}^{d})$, $\|f\|_{\ell^{2}(\mathbb{Z}^{d})}+\|g\|_{\ell^{\frac{2d}{d+2}}(\mathbb{Z}^{d})}<\delta$, then we have a unique strong solution $\widetilde{u}\in C(\mathbb{R};X)$ of the equation (\ref{DNLW}), with the initial data $F_{0}$. Furthermore, there exists $F_{\pm}\in X$, satisfying all assertions in Theorem \ref{S1}.
		\end{thm}
		\begin{thm}\label{S3}
			For $3\le d\le 5$, $p\ge \frac{d(\beta_{d}+2)}{\beta_{d}(d-2)}$, there exists $\delta>0$, such that if $F_{0}\in X$, $\|F_{0}\|_{X}<\delta$, then we have a unique strong solution $\widetilde{u}\in C(\mathbb{R};X)$ of the equation (\ref{DNLW}), with the initial data $F_{0}$. Furthermore, there exists $F_{\pm}\in X$, satisfying all assertions in Theorem \ref{S1.5}.
		\end{thm}
		\begin{rem}
			Compared with Theorem \ref{S2} and Theorem \ref{S3}, Theorem \ref{S1} and Theorem \ref{S1.5} are only suitable for the defocusing case. However, we all require some smallness conditions in proving the asymptotic completeness. In fact, based on Remark \ref{3.18}, the solution will blow up in the focusing case, without smallness condition, which directly breaks the asymptotic completeness. 
		\end{rem}
		\vspace{10pt}
		\subsection{Organization and notations}
		We organize the paper as follows. In Section 2, we establish the Strichartz estimate by reducing it to the decay estimate of the oscillatory integrals. We then present the most recent advances in this direction, which are essential for satisfying the technical conditions required in \cite{4}. In Section 3, we will introduce two abstract frameworks in \cite{4}. In Section 4, we use one of the frameworks to establish our desired scattering theory, giving the proof of Theorem \ref{main}, \ref{main2}, \ref{main3}. In Section 5, we reconstruct the scattering theory by proving Theorem \ref{S1}, \ref{S1.5}, \ref{S2}, \ref{S3}, with the Strichartz estimate.

		\par 
		Finally, we review some frequently-used notations, where $I$ is a time interval.
		\begin{itemize}
			\item By $u\in C^{k}(I; B)$ (resp. $L^{p}(I;B)$) for a Banach space $B,$ we mean $u$ is a $C^{k}$ (resp. $L^{p}$) map from $I$ to $B;$ see page 301 in \cite{32}.
			\item By $u\in C_{0}(\Z^{d}\times I)$, we mean $u$ has a compact support on $\Z^{d}\times I$.
			\item By $A\lesssim B$ (resp. $A\approx B$), we mean there is a positive constant $C$, such that $A\le CB$ (resp. $C^{-1}B\le A \le C B$). If the constant $C$ depends on $p,$ then we write $A\lesssim_{p}B$ (resp. $A\approx_{p} B$).
			\item By $T\in \mathcal{B}(U,V)$ for the normed spaces $U,V$, we mean $T:U\to V$ is a bounded linear map.
			\item Set $\langle m \rangle:=(1+|m|^{2})^{1/2}$ and $|m|:=(\sum_{j=1}^{d}|m_{j}|^{2})^{1/2}$ for $m=(m_{1},\cdots, m_{d})\in \Z^{d}$.
			\item By $u\in\mathcal{O}(\Omega)$ for a region $\Omega$, we mean $u$ is a holomorphic function on $\Omega$.
			\item By $(f)_{+}$ for a real valued function $f$, we refer to the nonnegative part of $f$.
		\end{itemize}

		\section{The uniform decay estimate of DW}
		In this section, we will show the relationship between the oscillatory integral theory and the uniform decay estimate of DW, and summarize the uniform decay estimate obtained in \cite{1,2,3}. And we use the Newton polyhedron method to reprove the uniform decay estimate in 3D. Relatively speaking, it will be more concise than the proof in \cite{1}.
		
		\subsection{Basic tools in DW}
		We now focus on the following DW with the vanishing initial displacement.
		\begin{equation}\label{DLW}
			\left\{
			\begin{aligned}
				& \partial_{t}^{2} u(x,t)-\Delta u(x,t) =0,  \\
				& u(x,0)=0, \partial_{t}u(x,0)=g(x),\quad (x,t)\in \mathbb{Z}^d\times \mathbb{R}.
			\end{aligned}
			\right.
		\end{equation}
		\begin{rem}\label{rem}
			In fact, it can be seen in the subsequent proof that the uniform decay estimate of DW is also valid for general displacement $f(x)$. As the key to such a decay estimate is the nature of the phase function, which does not change in the general case. For the convenience of description, we only consider such a special case. Furthermore, $\partial_{t}u$ has the same uniform decay estimate as $u$, since $\partial_{t}u$ also satisfies the DW, with the initial displacement $g(x)$ and the vanishing initial velocity.
		\end{rem}
		
		To express the solution of DW in the form of the oscillatory integrals, we shall introduce the discrete Fourier transform and its inverse. 
		\begin{defi}
			For $u\in \ell^1(\mathbb{Z}^d)$, $g\in L^{1}(\mathbb{T}^{d})$, the discrete Fourier transform $\mathcal{F}$ and its inverse $\mathcal{F}^{-1}$ are defined as
			\begin{equation*}
				\mathcal{F}(u)(x):=\sum_{k\in \mathbb{Z}^{d}}u(k)e^{-ikx}, \quad \forall x\in \mathbb{T}^{d},
			\end{equation*}
			\begin{equation*}
				\mathcal{F}^{-1}(g)(k):=\frac{1}{(2\pi)^{d}}\int_{\mathbb{T}^d} g(x)e^{ikx}dx, \quad \forall k\in \mathbb{Z}^{d},
			\end{equation*}
			where $\mathbb{T}^{d}$ is the $d$-dimensional torus, parameterized as $[-\pi,\pi)^d$. From basic knowledge in the Fourier series, $\mathcal{F}$ can be extended to an isomorphism between $\ell^{2}(\mathbb{Z}^{d})$ and $L^{2}(\mathbb{T}^{d})$.
		\end{defi} 
		Next we recall the definition of the convolution on $\Z^{d}$ and the Young inequality. 
		\begin{defi}
			For $f\in\ell^{p}(\mathbb{Z}^{d})$, $g\in\ell^{q}(\mathbb{Z}^{d})$, where $1\le p,q\le\infty$, $\frac{1}{p}+\frac{1}{q}=1$, we define the convolution of $f$ and $g$ as
			\begin{equation*}
				(f\ast g)(x):=\sum_{y\in\mathbb{Z}^{d}}f(y)g(x-y), \quad x\in \mathbb{Z}^{d}.
			\end{equation*}
		\end{defi}

		\begin{thm}\label{Young}
			For $1\le p,q,r\le\infty$, satisfying
			\begin{equation*}
				1+\frac{1}{q}=\frac{1}{p}+\frac{1}{r},
			\end{equation*}
			we have the following Young inequality for $f\in \ell^{p}(\mathbb{Z}^{d}), g\in \ell^{r}(\mathbb{Z}^{d})$.
			\begin{equation*}
				\|f\ast g\|_{\ell^{q}(\mathbb{Z}^{d})}\lesssim_{p,q,r}\|f\|_{\ell^{p}(\mathbb{Z}^{d})}\|g\|_{\ell^{r}(\mathbb{Z}^{d})}.
			\end{equation*}
		\end{thm}
		\vspace{10pt}
		\subsection{Summary of the uniform decay estimate}
		
		Now the uniform decay estimate of DW can be reduced to some oscillatory integrals. In fact, applying the discrete Fourier transform to the equation (\ref{DLW}), we can derive $u=G(\cdot,t)\ast g$, where
		\begin{equation}\label{Green}
			G(x,t):=\frac{1}{(2\pi)^{d}}\int_{\mathbb{T}^d}e^{ix\cdot\xi }\frac{\sin(t\omega(\xi))}{\omega(\xi)}d\xi, \quad \omega(\xi):=\left(\sum_{j=1}^{d}2-2\cos(\xi_{j})\right)^{\frac{1}{2}}.
		\end{equation}
		From a simple observation, we have $G(x,t)=-$Im $I(v,t)$, $x=vt$,
		\begin{equation}\label{Oscillatory}
			I(v,t):=\frac{1}{(2\pi)^{d}}\int_{\mathbb{T}^d}e^{it\phi(v,\xi) }\frac{1}{\omega(\xi)}d\xi, \quad \phi(v,\xi):=v\cdot\xi-\omega(\xi).
		\end{equation}
		Using the Young inequality, we also have 
		\begin{equation*}
			\sup_{x\in\mathbb{Z}^{d}}|u(x,t)|=\|u(\cdot,t)\|_{\ell^{\infty}(\mathbb{Z}^{d})}\lesssim \|g\|_{\ell^{1}(\mathbb{Z}^{d})}\|G(\cdot,t)\|_{\ell^{\infty}(\mathbb{Z}^{d})}\le\|g\|_{\ell^{1}(\mathbb{Z}^{d})}\cdot\sup_{v\in\mathbb{R}^{d}}|I(v,t)|.
		\end{equation*}
		Then the whole problem is reduced to the uniform decay estimate of the oscillatory integral $I(v,t)$.
		
		Based on the results in  \cite{1,2,3}, we can summarize the sharp uniform decay estimate as follows.
		\noindent
		\begin{itemize}
			\item $d=2$ : $\sup_{v\in \mathbb{R}^{d}}|I(v,t)|\lesssim (1+|t|)^{-\frac{3}{4}}=\langle t\rangle^{-\frac{3}{4}}$.
			\item $d=3$ : $\sup_{v\in \mathbb{R}^{d}}|I(v,t)|\lesssim (1+|t|)^{-\frac{7}{6}}=\langle t\rangle^{-\frac{7}{6}}$.
			\item $d=4$ : $\sup_{v\in \mathbb{R}^{d}}|I(v,t)|\lesssim (1+|t|)^{-\frac{3}{2}}\log(2+|t|)=\langle t\rangle^{-\frac{3}{2}}\log(2+|t|)$.
			\item $d=5$ : $\sup_{v\in \mathbb{R}^{d}}|I(v,t)|\lesssim (1+|t|)^{-\frac{11}{6}}=\langle t\rangle^{-\frac{11}{6}}$.
		\end{itemize}
		Directly, we also have 
		\noindent
		\begin{itemize}
			\item $d=2$ : $\|u(\cdot,t)\|_{\ell^{\infty}(\mathbb{Z}^{d})}\lesssim (1+|t|)^{-\frac{3}{4}}\|g\|_{\ell^{1}(\mathbb{Z}^{d})}=\langle t\rangle^{-\frac{3}{4}}\|g\|_{\ell^{1}(\mathbb{Z}^{d})}$.
			\item $d=3$ : $\|u(\cdot,t)\|_{\ell^{\infty}(\mathbb{Z}^{d})}\lesssim (1+|t|)^{-\frac{7}{6}}\|g\|_{\ell^{1}(\mathbb{Z}^{d})}=\langle t\rangle^{-\frac{7}{6}}\|g\|_{\ell^{1}(\mathbb{Z}^{d})}$.
			\item $d=4$ : $\|u(\cdot,t)\|_{\ell^{\infty}(\mathbb{Z}^{d})}\lesssim (1+|t|)^{-\frac{3}{2}}\log(2+|t|)\|g\|_{\ell^{1}(\mathbb{Z}^{d})}=\langle t\rangle^{-\frac{3}{2}}\log(2+|t|)\|g\|_{\ell^{1}(\mathbb{Z}^{d})}$.
			\item $d=5$ :  $\|u(\cdot,t)\|_{\ell^{\infty}(\mathbb{Z}^{d})}\lesssim (1+|t|)^{-\frac{11}{6}}\|g\|_{\ell^{1}(\mathbb{Z}^{d})}=\langle t\rangle^{-\frac{11}{6}}\|g\|_{\ell^{1}(\mathbb{Z}^{d})}$.
		\end{itemize}
		\vspace{5pt}
		To derive more general uniform decay estimate of DW, we need the following simple lemma.
		\begin{lemma}\label{lemma}
			Considering the following integral $\Omega(t)$,
			\begin{equation*}
				\Omega(t):=\int_{\mathbb{T}^d}\frac{\sin^{2}(t\omega(\xi))}{\omega(\xi)^{2}} d\xi, \quad \omega(\xi)=\left(\sum_{j=1}^{d}2-2\cos(\xi_{j})\right)^{\frac{1}{2}}=\left(\sum_{j=1}^{d}4\sin^{2}(\frac{\xi_{j}}{2})\right)^{\frac{1}{2}}.
			\end{equation*}
			\begin{itemize}
				\item if $d\ge3$, then $\Omega(t)\lesssim _{d} 1$.
				\item if $d=2$, then $\Omega(t)\lesssim_{d} \log(t+2)$.
			\end{itemize}
		\end{lemma}
		\begin{proof}
			As $\omega(\xi)$ has a removable singularity at $\xi=0$, we decompose the integral into two parts.
			\begin{equation*}
				\Omega(t)=\int_{B_{\delta}(0)}\frac{\sin^{2}(t\omega(\xi))}{\omega(\xi)^{2}} d\xi+\int_{\mathbb{T}^{d}-B_{\delta}(0)}\frac{\sin^{2}(t\omega(\xi))}{\omega(\xi)^{2}}, d\xi:=\Omega_{1}(t)+\Omega_{2}(t),
			\end{equation*}
			where $\delta\ll1$. Directly, we have
			\begin{equation*}
				\Omega_{2}(t)\le \int_{\mathbb{T}^{d}-B_{\delta}(0)}\frac{1}{\omega(\xi)^{2}} d\xi\lesssim1.
			\end{equation*}
			For $\Omega_{1}(t)$, we change the variables as follows.
			\begin{equation}\label{change}
				\zeta_{j}:=2\sin(\frac{\xi_{j}}{2}), \; j=1,2,\cdots,d.
			\end{equation}
			Then we derive that
			\begin{equation*}
				\Omega_{1}(t)=\int_{U}\frac{\sin^{2}(t|\zeta|)}{|\zeta|^{2}}J(\zeta)d\zeta,
			\end{equation*}
			where $U$ is a sufficiently small neighborhood of $0$, $J(\zeta)$ is the corresponding Jacobi term. Then we have
			\begin{equation*}
				\Omega_{1}(t)=\int_{U}\frac{\sin^{2}(t|\zeta|)}{|\zeta|^{2}}J(\zeta)d\zeta\lesssim\int_{B_{1}(0)}\frac{\sin^{2}(t|\zeta|)}{|\zeta|^{2}}d\zeta
			\end{equation*}
			\begin{equation*}
				=\int_{0}^{1}d\rho\int_{\mathbb{S}_{\rho}}\frac{\sin^{2}(t|\zeta|)}{|\zeta|^{2}}dS\lesssim \int_{0}^{1}\frac{\sin^{2}(t\rho)}{\rho^{3-d}}d\rho.
			\end{equation*}
			If $d\ge3$, we see $\Omega_{1}(t)\lesssim 1$. For $d=2$, let $r:=t\rho$, we have
			\begin{equation*}
				\Omega_{1}(t)\lesssim \int_{0}^{1}\frac{\sin^{2}(t\rho)}{\rho}d\rho=\int_{0}^{t}\frac{\sin^{2}(r)}{r}dr\lesssim \log(2+t).
			\end{equation*}
		\end{proof}
		Finally, we use the Young inequality and derive the following general uniform decay estimate, where $\beta_{d}$ can be referred to Section 1.
		\begin{thm}\label{jj}
			Suppose $u$ is a solution of DW (\ref{DLW}), then for $d=2,3,4,5$, $1\le p,q\le\infty$, $k>2$,
			\begin{equation*}
				1+\frac{1}{q}=\frac{1}{k}+\frac{1}{p},
			\end{equation*}
			we have the following general uniform decay estimate
			\begin{equation*}
				\|u(\cdot,t)\|_{\ell^{q}(\mathbb{Z}^{d})}\lesssim (1+|t|)^{-\beta_{d}(1-\frac{2}{k})}\|g\|_{\ell^{p}(\mathbb{Z}^{d})}.
			\end{equation*}
		\end{thm}
		\begin{proof}
			For convenience, we only deal with the case $d=3,4,5$ and the proof of the case $d=2$ is similar. Based on the uniform decay estimate established above and Lemma \ref{lemma}, we can now derive that
			\begin{equation*}
				\|G(\cdot,t)\|_{\ell^{\infty}(\mathbb{Z}^{d})}\lesssim (1+|t|)^{-\beta_{d}}, \quad \|G(\cdot,t)\|_{\ell^{2}(\mathbb{Z}^{d})}\lesssim 1.
			\end{equation*}
			From the interpolation, we see that
			\begin{equation*}
				\|G(\cdot,t)\|_{\ell^{k}(\mathbb{Z}^{d})}\lesssim (1+|t|)^{-\beta_{d}(1-\frac{2}{k})}, \quad k\ge2.
			\end{equation*}
			Using the Young inequality, we can complete the proof as follows.
			\begin{equation*}
				\|u(\cdot,t)\|_{\ell^{q}(\mathbb{Z}^{d})}=\|G(\cdot,t)\ast g\|_{\ell^{q}(\mathbb{Z}^{d})}\lesssim \|G(\cdot,t)\|_{\ell^{k}(\mathbb{Z}^{d})}\|g\|_{\ell^{p}(\mathbb{Z}^{d})}\lesssim (1+|t|)^{-\beta_{d}(1-\frac{2}{k})}\|g\|_{\ell^{p}(\mathbb{Z}^{d})}.
			\end{equation*}
		\end{proof}
		\vspace{10pt}
		\subsection{Uniform decay estimate of DW in 3 dimension}
		In the following content, we obey the notations in \cite{2}, and reprove the 3-dimensional uniform decay estimate in \cite{1}, with Newton polyhedron method. Before proceeding with the content below, it is recommended to familiarize oneself with some notations and properties of the Newton polyhedra and the uniform decay estimate presented in the Appendix.
		
		Technically, the integral (\ref{Oscillatory}) is on $d$-torus $\mathbb{T}^{d}$, but the oscillatory integral theory is always set in $\mathbb{R}^{d}$. To overcome this, we need the following lemma.
		\begin{lemma}\label{partition}
			There exists $\eta\in C_{c}^{\infty}(\mathbb{R}^{d})$, satisfying 
			\begin{equation}\label{q}
				\sum_{n\in\mathbb{Z}^{d}}\eta(\xi+2\pi n)=1, \quad \forall \xi\in \mathbb{R}^{d}.
			\end{equation}
		\end{lemma}
		\begin{proof}
		At first, we consider the case $d=1$. Take $\psi\in C_{c}^{\infty}(\mathbb{R})$, we define 
		\begin{equation*}
			\eta(\xi):=(\psi\ast\theta)(\xi), \quad \theta(\xi):=\left(1-\frac{|x|}{2\pi}\right)_{+}.
		\end{equation*}
		From the graph of $\theta(\xi)$, we see that $\sum_{n\in \Z}\theta(\xi+2\pi n):=\sum_{n\in \Z}\theta_{n}(\xi)=1, \forall \xi\in \mathbb{R}$. Thus, 
		\begin{equation*}
			\sum_{n\in\mathbb{Z}^{d}}\eta(\xi+2\pi n)=\sum_{n\in\mathbb{Z}^{d}}(\psi\ast\theta)(\xi+2\pi n)=\sum_{n\in\mathbb{Z}^{d}}(\psi\ast\theta_{n})(\xi)=1, \quad \forall \xi\in \mathbb{R}.
		\end{equation*}
		For general $d$, let 
		\begin{equation*}
			\eta(\xi):=\eta(\xi_{1})\eta(\xi_{2})\cdots\eta(\xi_{d}), \quad \xi=(\xi_{1},\xi_{2},\cdots,\xi_{d}).
		\end{equation*}
		Then for $n=(n_{1},n_{2}\cdots n_{d})$, 
		\begin{equation*}
			\sum_{n\in\mathbb{Z}^{d}}\eta(\xi+2\pi n)=\sum_{n_{1},n_{2},\cdots, n_{d}\in\mathbb{Z}}\eta_{n_{1}}(\xi_{1})\eta_{n_{2}}(\xi_{2})\cdots\eta_{n_{d}}(\xi_{d})=1, \quad \forall \xi\in \mathbb{R}^{d},
		\end{equation*}
		which completes the proof.
		\end{proof}
		
		Now we can easily transform the integral (\ref{Oscillatory}) into the integral on $\mathbb{R}^{d}$ as follows.
		\begin{equation}\label{Rd}
			I(v,t)=\frac{1}{(2\pi)^{d}}\sum_{n\in \mathbb{Z}^{d}}\int_{\mathbb{T}^d}e^{it\phi(v,\xi)}\frac{\eta(\xi+2\pi n)}{\omega(\xi)}d\xi=\frac{1}{(2\pi)^{d}}\int_{\mathbb{R}^d}e^{it\phi(v,\xi)}\frac{\eta(\xi)}{\omega(\xi)}d\xi.
		\end{equation}
		Based on the stationary phase method, we focus on the degenerate points of the phase function $\phi(v,\xi)$.
		
		From the simple calculation, we see the following expressions of the gradient and the Hessian.
		\begin{itemize}
			\item 
			\begin{equation*}
				\nabla\omega(\xi)=\frac{1}{\omega(\xi)}\left(\sin(\xi_{1}),\cdots, \sin(\xi_{d})\right), \quad \xi\in \mathbb{T}^{d}-\lbrace 0\rbrace.
			\end{equation*}
			\item 
			\begin{equation*}
				Hess_{\xi}\;\omega(\xi):=\left(\omega_{ij}(\xi)\right)_{i,j=1}^{d}=\left(\frac{-\sin(\xi_{i})\sin(\xi_{j})+\delta_{ij}\cos(\xi_{i})\omega(\xi)^{2}}{\omega(\xi)^{3}}\right)_{i,j=1}^{d}, \quad \xi\in \mathbb{T}^{d}-\lbrace 0\rbrace,
			\end{equation*}
			where $\delta_{ij}$ is the Kronecker function.
		\end{itemize}
		Notice that $|\sin(x)|\le2|\sin(\frac{x}{2})|, \; \forall x\in [-\pi,\pi]$, then
		\begin{equation*}
			|\nabla\omega(\xi)|^{2}=\frac{\sum_{j=1}^{d}\sin^{2}(\xi_{j})}{\sum_{j=1}^{d}4\sin^{2}(\frac{\xi_{j}}{2})}\le 1.
		\end{equation*}
		Thus for $v\in B_{1}(0)^{c}$, we have $\nabla\phi(v,\xi)\ne 0$, i.e. $\phi$ has no critical point.
		
		To further classify the degenerate points corresponding to their ranks, we need the following lemma.
		\begin{lemma}\label{det}
			$\xi$ is a degenerate point of $\omega$, if and only if 
			\begin{equation}\label{w}
				\prod_{i=1}^{d}\cos(\xi_{i})-\frac{1}{\omega(\xi)^{2}}\sum_{i=1}^{d}\left(\sin^{2}(\xi_{i})\prod_{j\ne i}\cos(\xi_{j})\right)=0,
			\end{equation}
			Equivalently, $\xi$ belongs to one of the following sets.
			\begin{itemize}
				\item $\Gamma_{1}:=\left\lbrace \xi\in\mathbb{T}^{d}-\lbrace0\rbrace \Big| 2d=\sum_{i=1}^{d}\sec(\xi_{i})+\cos(\xi_{i})\right\rbrace.$
				\item $\Gamma_{k}:=\left\lbrace \xi\in\mathbb{T}^{d}-\lbrace0\rbrace \Big| \xi\; \text{has exactly}\; k \text{components equal to}\;\pm\frac{\pi}{2}\right\rbrace$, \quad $k=2,\cdots,d.$
			\end{itemize}
		\end{lemma}
		
		\begin{proof}
			From previous calculation, we can reformulate it as follows.
			\begin{equation*}
				Hess_{\xi}\;\omega(\xi)=\frac{1}{\omega(\xi)}
				\left(\begin{bmatrix}
					\cos(\xi_{1}) & &  \\
					
					&\ddots& \\
					&&\cos(\xi_{d}) \\
					
				\end{bmatrix}-
				\frac{1}{\omega(\xi)^{2}}
				\begin{bmatrix}
					\sin(\xi_{1})\\
					
					\vdots \\
					\sin(\xi_{d})\\
				\end{bmatrix} \cdot
				\begin{bmatrix}
					\sin(\xi_{1})&\cdots&\sin(\xi_{d})\\
				\end{bmatrix}\right).
			\end{equation*}
			Using the perturbation method and the fact that, for $A\in M_{m\times m}$, $D\in M_{n\times n}$, $B\in M_{m\times n}$, $C\in M_{n\times m}$, with $A,D$ invertible, we have
			\begin{equation*}
				\left|A\right|\cdot\left|D-CA^{-1}B\right|=\left|\begin{bmatrix}
					A & B \\
					
					C & D \\
					
				\end{bmatrix}
				\right|=\left|D\right|\cdot\left|A-BD^{-1}C\right|.
			\end{equation*}
			We see the determinant of the matrix in the bracket is equal to LHS of (\ref{w}). Then the classification above follows. 
		\end{proof}
		
		Now we can get the following sharp uniform decay estimate, with a similar change of variables in \cite{8} and the Newton polyhedron method.
		\begin{thm}\label{uniform}
			For $d=3$, $\xi^{\ast}\in \mathbb{T}^{3}-\lbrace0\rbrace$, $v_{\xi^{\ast}}=\nabla\omega(\xi^{\ast})$, we have the uniform decay estimate as follows.
			\begin{itemize}
				\item (I):  if $\xi^{\ast}$ is nondegenerate, then $M(\Phi_{v_{\xi^{\ast}}},\xi^{\ast})\curlyeqprec(-\frac{3}{2},0).$
				\item (II): if $\xi^{\ast}\in \Gamma_{3}$, then $M(\Phi_{v_{\xi^{\ast}}},\xi^{\ast})\curlyeqprec(-\frac{7}{6},0).$
				\item (III): if $\xi^{\ast}\in \Gamma_{2}$, then $M(\Phi_{v_{\xi^{\ast}}},\xi^{\ast})\curlyeqprec(-\frac{5}{4},0).$
				\item (IV):  if $\xi^{\ast}\in \Gamma_{1}$, then $M(\Phi_{v_{\xi^{\ast}}},\xi^{\ast})\curlyeqprec(-\frac{4}{3},0).$
			\end{itemize}
		\end{thm}
		\begin{proof}
			For (I), considering the change of variables $\xi:=A\zeta+\xi^{\ast}$, where $A$ satisfies
			\begin{equation*}
				A^{T}\cdot Hess_\xi \;\omega(\xi^{\ast})\cdot A= diag\lbrace c_{1},c_{2},c_{3}\rbrace, \; c_{i}\ne 0, \forall i=1,2,3.	 		
			\end{equation*}
			Then, the Taylor series of $\Phi_{v_{\xi^{\ast}}}$ at $0$ is  
			\begin{equation*}
				\Phi_{v_{\xi^{\ast}}}(A\zeta+\xi^{\ast})=\left(v_{\xi^{\ast}}\cdot\xi^{\ast}-\omega(\xi^{\ast})\right)+\left(c_{1}\zeta_{1}^{2}+c_{2}\zeta_{2}^{2}+c_{3}\zeta_{3}^{2}\right)+O(|\zeta|^{3})
			\end{equation*}
			\begin{equation*}
				:=r_{\ast}+Q(\zeta)+O(|\zeta|^{3}).
			\end{equation*}
			Let $\gamma:=(\frac{1}{2},\frac{1}{2},\frac{1}{2})$, we see $Q(\zeta)\in \mathcal{E}_{\gamma,3}$, $O(|\zeta|^{3})\in H_{\gamma,3}$. Using Lemma \ref{we} and Lemma \ref{Q} in the Appendix, we derive that 
			\begin{equation*}
				M(\Phi_{v_{\xi^{\ast}}},\xi^{\ast})\curlyeqprec M(c_{1}\zeta_{1}^{2}+c_{2}\zeta_{2}^{2}+c_{3}\zeta_{3}^{2})
			\end{equation*}
			\begin{equation*}
				\curlyeqprec(-\frac{1}{2},0)+(-\frac{1}{2},0)+(-\frac{1}{2},0)\curlyeqprec(-\frac{3}{2},0).
			\end{equation*}
			
			For (II), without loss of generality, we suppose that $\xi^{\ast}$=$(\frac{\pi}{2},\frac{\pi}{2},\frac{\pi}{2})$. Considering the change of variables $\xi:=A\zeta+\xi^{\ast}$, where $A$ is defined as follows.
			\begin{equation*}
				A=
				\begin{bmatrix}
					1&0&-1\\
					0&1&-1\\
					0&0&1\\
				\end{bmatrix}.
			\end{equation*}
			Then Taylor series of $\Phi_{v_{\xi^{\ast}}}$ is 
			\begin{equation*}
				\Phi_{v_{\xi^{\ast}}}(A\zeta+\xi^{\ast})=\left(v_{\xi^{\ast}}\cdot\xi^{\ast}-\omega(\xi^{\ast})\right)+ (a\zeta_{3}^{2}+b\zeta_{1}^{2}\zeta_{2}+c\zeta_{1}\zeta_{2}^{2})+R(\zeta)
			\end{equation*}
			\begin{equation*}
				:=r_{\ast}+Q(\zeta)+R(\zeta),
			\end{equation*}
			where $R(\zeta)$ has no $\zeta_{1}^{2}, \zeta_{2}^{2}$-term, and $a,b,c\ne0$. Let $\gamma:=(\frac{1}{3},\frac{1}{3},\frac{1}{2})$, we have $Q(\zeta)\in \mathcal{E}_{\gamma,3}$, $R(\zeta)\in H_{\gamma,3}$. Using Lemma \ref{we}, Lemma \ref{Q} and Lemma \ref{-2/3}, we derive that
			\begin{equation*}
				M(\Phi_{v_{\xi^{\ast}}},\xi^{\ast})\curlyeqprec M(a\zeta_{3}^{2}+b\zeta_{1}^{2}\zeta_{2}+c\zeta_{1}\zeta_{2}^{2})
			\end{equation*}
			\begin{equation*}
				\curlyeqprec(-\frac{1}{2},0)+M(b\zeta_{1}^{2}\zeta_{2}+c\zeta_{1}\zeta_{2}^{2})\curlyeqprec(-\frac{1}{2},0)+(-\frac{2}{3},0)\curlyeqprec(-\frac{7}{6},0).
			\end{equation*}
			
			For (III), we can suppose that $\xi^{\ast}=(\frac{\pi}{2},\frac{\pi}{2},\xi_{0})$, where $\xi_{0}\ne \frac{\pi}{2}$.
			Considering the change of variables $\xi:=A\zeta+\xi^{\ast}$, where $A$ is defined as follows.
			\begin{equation*}
				A=
				\begin{bmatrix}
					1&0&0\\
					-1&1&0\\
					0&0&1\\
				\end{bmatrix}.
			\end{equation*}
			Then we have the following Taylor series.
			\begin{equation*}
				\Phi_{v_{\xi^{\ast}}}(A\zeta+\xi^{\ast})=\left(v_{\xi^{\ast}}\cdot\xi^{\ast}-\omega(\xi^{\ast})\right)+ (a\zeta_{2}^{2}+b\zeta_{2}\zeta_{3}+c\zeta_{3}^{2}+d\zeta_{1}^{2}\zeta_{2}+\lambda\zeta_{1}^{4})+R(\zeta)
			\end{equation*}
			\begin{equation*}
				:=r_{\ast}+Q(\zeta)+R(\zeta),
			\end{equation*}
			where $R(\zeta)$ has no $\zeta_{1}^{2},\zeta_{1}^{3}$-term, with $d\ne0$, $b^{2}\ne 4ac$ ($\lambda$ is uncertain). Then we rotate $\zeta_{2}\zeta_{3}$-plane appropriately to cancel $\zeta_{2}\zeta_{3}$-term. Let $\gamma:=(\frac{1}{4},\frac{1}{2},\frac{1}{2})$, we have $Q(\zeta)\in \mathcal{E}_{\gamma,3}$, $R(\zeta)\in H_{\gamma,3}$. Using lemmas in the Appendix again, we deduce that
			\begin{equation*}
				M(\Phi_{v_{\xi^{\ast}}},\xi^{\ast})\curlyeqprec M(a\zeta_{2}^{2}+c\zeta_{3}^{2}+d\zeta_{1}^{2}\zeta_{2}+\lambda\zeta_{1}^{4})
			\end{equation*}
			\begin{equation*}
				\curlyeqprec(-\frac{1}{2},0)+M(a\zeta_{2}^{2}+d\zeta_{1}^{2}\zeta_{2}+\lambda \zeta_{1}^{4})\curlyeqprec(-\frac{1}{2},0)+(-\frac{3}{4},0)\curlyeqprec(-\frac{5}{4},0).
			\end{equation*}
			
			For (IV), we can also suppose that $\xi^{\ast}=(\xi_{1}^{\ast},\xi_{2}^{\ast},\xi_{3}^{\ast})$. Considering the change of variables $\xi:=A\zeta+\xi^{\ast}$, where $A$ is defined as follows.
			\begin{equation*}
				A=
				\begin{bmatrix}
					\tan(\xi_{1}^{\ast})&-\tan(\xi_{2}^{\ast})&-\tan(\xi_{3}^{\ast})\\
					\tan(\xi_{2}^{\ast})&\tan(\xi_{1}^{\ast})&0\\
					\tan(\xi_{3}^{\ast})&0&\tan(\xi_{1}^{\ast})\\
				\end{bmatrix}.
			\end{equation*}
			Similarly, we have 
			\begin{equation*}
				\Phi_{v_{\xi^{\ast}}}(A\zeta+\xi^{\ast})=\left(v_{\xi^{\ast}}\cdot\xi^{\ast}-\omega(\xi^{\ast})\right)+ (a\zeta_{2}^{2}+b\zeta_{2}\zeta_{3}+c\zeta_{3}^{2}+d\zeta_{1}^{3})+R(\zeta)
			\end{equation*}
			\begin{equation*}
				:=r_{\ast}+Q(\zeta)+R(\zeta),
			\end{equation*}
			where $R(\zeta)$ contains no $\zeta_{1}^{2}$-term, with $d\ne0$, $b^{2}\ne 4ac$. Parallelly, we can rotate $\zeta_{2}\zeta_{3}$-plane, and cancel $\zeta_{2}\zeta_{3}$-term. Let $\gamma:=(\frac{1}{3},\frac{1}{2},\frac{1}{2})$, we have $Q(\zeta)\in \mathcal{E}_{\gamma,3}$, $R(\zeta)\in H_{\gamma,3}$. Then we see
			\begin{equation*}
				M(\Phi_{v_{\xi^{\ast}}},\xi^{\ast})\curlyeqprec M(a\zeta_{2}^{2}+c\zeta_{3}^{2}+d\zeta_{1}^{3})
			\end{equation*}
			\begin{equation*}
				\curlyeqprec(-\frac{1}{2},0)+(-\frac{1}{2},0)+(-\frac{1}{3},0)\curlyeqprec(-\frac{4}{3},0).
			\end{equation*}
		\end{proof}
		\begin{rem}
			In the proof of Theorem \ref{uniform}, we omit a rather complicate detail. For (IV), the statement ``$d\ne0$" is not obvious. We first need the following conclusion in \cite{17} to simplify the whole problem.
			\begin{itemize}
				\item For $m\ge2$,令$f\in C^{m+1}(U)$, $\xi_{0}\in U\subseteq \mathbb{R}^{d}$. If the Hessian $Hess_{\xi} \; f(\xi_{0})$ has rank $d-1$, $v_{d}$ is its null vector, and $\left(v_{d}\cdot\nabla_{\xi}\right)^{j}f(\xi_{0})=0$, $\forall 2\le j\le m$. Then  $\left(v_{d}\cdot\nabla_{\xi}\right)^{m+1}f(\xi_{0})=0$ is equivalent to $\left(v_{d}\cdot\nabla_{\xi}\right)^{m-1}\det Hess_{\xi}\; f(\xi_{0})=0$.
			\end{itemize}
			Based on this conclusion, we just need to show $\left(v_{d}\cdot\nabla_{\xi}\right)\det Hess_{\xi}\;\omega(\xi^{\ast})\ne0$, where $v_{d}=\left(\tan(\xi_{1}^{\ast}),\tan(\xi_{2}^{\ast}),\tan(\xi_{3}^{\ast})\right)$. Recall the calculation in Lemma \ref{det}, we have
			\begin{equation*}
				\det Hess_{\xi}\;\omega(\xi)=\frac{1}{\omega(\xi)^{5}}\left(\prod_{j=1}^{3}\cos(\xi_{j})\right)\left[\omega(\xi)^{2}-\sum_{j=1}^{3}\sin(\xi_{j})\cos(\xi_{j})\right]
			\end{equation*}
			\begin{equation*}
				:=J_{1}\cdot J_{2}\cdot J_{3}.
			\end{equation*}
			As $\xi^{\ast}\in\Gamma_{1}$, then we have $\sum_{j=1}^{3}\sec(\xi_{j}^{\ast})+\cos(\xi_{j}^{\ast})=6$. Simple calculation shows that the value of $J_{3}$ at $\xi^{\ast}$ is $0$, so we just need to calculate the derivatives on $J_{3}$. Then we have
			\begin{equation*}
				\left(v_{d}\cdot\nabla_{\xi}\right)\det Hess_{\xi} \; \omega(\xi^{\ast})\ne0 \iff \sum_{j=1}^{3}\sin(\xi_{j}^{\ast})\tan^{3}(\xi_{j}^{\ast})\ne 0.
			\end{equation*}
			Intuitively, from the graph \cite{23}, we see $\xi=0$ is the only solution of the following equation.
			\begin{equation*}
				\left\{
				\begin{aligned}
					&  \sum_{j=1}^{3}\sin(\xi_{j})\tan^{3}(\xi_{j})= 0, \\
					&  \sum_{j=1}^{3}\sec{(\xi_{j})}+\cos(\xi_{j})=6.
				\end{aligned}
				\right.
			\end{equation*}
			However, the rigorous proof is rather involved.
		\end{rem}
		With Theorem \ref{uniform} above, we can reprove the uniform decay estimate of 3-dimensional DW. As $\frac{1}{\omega}$ has singularity at $0$, we take the cutoff $\chi\in C_{c}^{\infty}(\mathbb{R}^{3})$, such that $\chi=1$ in a small neighborhood of $0$. From (\ref{Rd}), we have
		\begin{equation}\label{part}
			I(v,t)=\frac{1}{(2\pi)^{d}}\int_{\mathbb{R}^d}e^{it\phi(v,\xi)}\frac{\eta(\xi)\chi(\xi)}{\omega(\xi)}d\xi+\frac{1}{(2\pi)^{d}}\int_{\mathbb{R}^d}e^{it\phi(v,\xi)}\frac{\eta(\xi)(1-\chi(\xi))}{\omega(\xi)}d\xi
		\end{equation}
		\begin{equation*}
			:=I_{1}(v,t)+I_{2}(v,t).
		\end{equation*}
		Using the estimate in \cite{1}, we have 
		\begin{equation}\label{d-1}
			\sup_{v\in \mathbb{R}^{d}}|I_{1}(v,t)|\le C(1+|t|)^{-(d-1)}, \quad \forall t\in \mathbb{R}.
		\end{equation}
		Based on Theorem \ref{uniform}, for any $\xi^{*}\in supp \; \eta(1-\chi):=U$, there exists a neighborhood $U_{\xi^{\ast}}\ni\xi^{\ast}$, $V_{\xi^{\ast}}\ni v_{\xi^{\ast}}=\nabla\omega(\xi^{\ast})$, such that for any $\zeta\in C_{c}^{\infty}(U_{\xi^{\ast}})$, there exists $C=C(\xi^{\ast},\zeta)>0$, with
		\begin{equation*}
			\sup_{v\in V_{\xi^{\ast}}}\left|\int_{\mathbb{R}^d}e^{it\phi(v,\xi)}\zeta(\xi)d\xi\right|\le C(1+|t|)^{-\beta}, \quad \forall t\in \mathbb{R},
		\end{equation*}
		where $\beta\in\lbrace\frac{3}{2},\frac{7}{6},\frac{5}{4},\frac{4}{3}\rbrace$.
		If $v\in V_{\xi^{\ast}}^{c}$, then $|\nabla\phi(v,\xi)|$ has strict lower bound. From Lemma \ref{j}, we have the following uniform decay estimate:
		\begin{equation*}
			\sup_{v\in \mathbb{R}^{3}}\left|\int_{\mathbb{R}^d}e^{it\phi(v,\xi)}\zeta(\xi)d\xi\right|\le C(1+|t|)^{-\beta}, \quad \forall t\in \mathbb{R}.
		\end{equation*}
		As $\lbrace U_{\xi^{\ast}}\rbrace_{\xi^{\ast}\in U}$ form a open covering of $U$, there exists a finite sub-covering $\lbrace U_{\xi_{j}^{\ast}}\rbrace_{j=1}^{N}$, and a corresponding partition of unity:
		\begin{equation*}
			\sum_{j=1}^{N}\theta_{j}\equiv1,\quad \theta_{j}\in C_{c}^{\infty}(U_{\xi_{j}^{\ast}}).
		\end{equation*}
		Then we can deduce that
		\begin{equation*}
			|I_{2}(v,t)|\le\sum_{j=1}^{N}\frac{1}{(2\pi)^{d}}\left|\int_{\mathbb{R}^d}e^{it\phi(v,\xi)}\frac{\eta(\xi)(1-\chi(\xi))}{\omega(\xi)}\theta_{j}(\xi)d\xi\right|\lesssim (1+|t|)^{-\frac{7}{6}}.
		\end{equation*}
		Thus we derive the uniform decay estimate of DW in 3 dimension, with the Newton polyhedron method.
		\section{Some abstract frameworks of the nonlinear dispersive equations}
		In this section, we will introduce two abstract frameworks (Z and A) and the corresponding results in \cite{4}.

		\subsection{Abstract framework-Z}
		In the abstract framework-Z, we have the following assumptions on the equation (\ref{abstract}):
		\begin{itemize}
			\item (I) $X$ is a Hilbert space, with norm $|\cdot|_{2}$. $\lbrace U_{0}(t)\rbrace_{t\in \mathbb{R}}$ is an one-parameter unitary operator group on $X$.
			\item (II) $X_{1},X_{3}$ are Banach spaces, with norms $|\cdot|_{1}$, $|\cdot|_{3}$. And $P0=0$, maps a neighborhood of $0$ in $X_{3}$ to $X_{1}$, satisfying 
			\begin{equation*}
				|Pf-Pg|_{1}\le C(|f|_{3}+|g|_{3})^{p-1}|f-g|_{3}, \quad p>1.
			\end{equation*}
			where $C$ is a positive constant.
			\item (III) $X_{1},X,X_{3}$ can be embedded into a Banach space $X_{4}$, and $X\cap X_{3}$ is dense in $X,X_{3}$. We also define $Z=L^{p+1}(\mathbb{R};X_{3})\cap B(\mathbb{R};X_{3})$, which is the reason why this framework is called ``Z''.
			\item (IV) For any $f\in X$, $t\to U_{0}(t)f$ belongs to $L^{p+1}(\mathbb{R};X_{3})$. And $X$ can be embedded into $X_{3}$.
			\item (V)  $U_{0}(t)$, restricted to $X\cap X_{1}$, has an extension to $X_{1}$, which maps to $X_{3}$, with the decay estimate below.
			\begin{equation*}
				|U_{0}(t)f|_{3}\le \frac{C}{|t|^{\alpha}}|f|_{1}, \quad \forall f\in X_{1}, t\ne 0,
			\end{equation*}
			where $C$ is a positive constant, $\alpha=\frac{2}{p+1}$. Besides, $U_{0}(t)$, restricted to $X\cap X_{3}$, also has an extension to $X_{3}$, which maps to $X_{4}$ and retains the group operations, i.e.
			\begin{equation*}
				U_{0}(t)U_{0}(s)f=U_{0}(t+s)f, \quad \forall f\in X_{1}.
			\end{equation*}
			\item (VI)  $G$ is a functional, which maps a neighborhood of $0$ in $X_{3}$ to $\mathbb{R}$. Besides, $G$ is lower semi-continuous and continuous at $0$.
			\item (VII) For any time interval $I$, $s\in I$, $f\in X$, if $u\in Z$, $\|u\|_{Z}$ is sufficiently small, and satisfies ($\ast_{s},f$) in $I$, then we have $u\in C(I;X)$, with the energy conservation
			\begin{equation}\label{independent}
				\frac{1}{2}|u(t)|_{2}^{2}+G(u(t)) \; \text{is independent with $t$.}
			\end{equation}
		\end{itemize}
		Then there are some scattering results of the equation (\ref{abstract}):
		\begin{thm}\label{scattering}
			We suppose the assumptions (I)$\sim$(VII), then there exists $\delta>0$, such that if $f_{-}\in X$, $|f_{-}|_{2}<\delta$, we have a unique $u$ satisfying the equation ($\ast_{-\infty},f_{-}$), i.e.
			\begin{equation*}
				u(t)=U_{0}(t)f_{\pm}+\int_{\pm\infty}^{t}U_{0}(t-s)Pu(s)ds , \quad \forall t\in \mathbb{R},
			\end{equation*}
			with $u\in C(\mathbb{R};X)\cap L^{p+1}(\mathbb{R};X_{3})$. Furthermore, we have
			\begin{equation}\label{asy}
				|u(t)-U_{0}(t)f_{-}|_{2}\to 0, \quad t\to-\infty,
			\end{equation}
			and $f_{+}\in X$, such that $u$ satisfies the equation ($\ast_{+\infty},f_{+}$),
			\begin{equation*}
				|u(t)-U_{0}(t)f_{+}|_{2}\to 0, \quad t\to+\infty,
			\end{equation*}
			and the energy conservation
			\begin{equation*}
				\frac{1}{2}|u(t)|_{2}^{2}+G(u(t))=\frac{1}{2}|f_{-}|_{2}^{2}=\frac{1}{2}|f_{+}|_{2}^{2}.
			\end{equation*}
		\end{thm}
		\begin{rem}
			The above theorem establishes the existence of the wave operator and the scattering operator for the equation (\ref{abstract}) with small initial data. 
		\end{rem}
		\begin{thm}\label{theorem3}
			We suppose the assumptions (I)$\sim$(VII), and $f_{-}\in X$, then there exists $T>-\infty$, and a unique solution $u\in C((-\infty,T];X)\cap L^{p+1}((-\infty,T];X_{3})$, satisfying the equation ($\ast_{-\infty},f_{-}$) in $(-\infty,T]$, such that we have (\ref{asy}) and 
			\begin{equation}\label{f}
				\frac{1}{2}|u(t)|_{2}^{2}+G(u(t))=\frac{1}{2}|f_{-}|_{2}^{2}.
			\end{equation}
		\end{thm}
		\begin{rem}
			This theorem guarantees that if $f_{-}$ is not sufficiently small, we can also establish the well-posedness on $(-\infty,T]$, retaining the energy conservation.
		\end{rem}
		\begin{thm}\label{wee}
			We also suppose the assumptions (I)$\sim$(VII), and $f_{0}\in X$, $|f_{0}|_{2}<\delta$, then there exists a unique strong solution $u$ of the equation (\ref{abstract}) in $\mathbb{R}$ with the initial data. And $u$ also satisfies the energy conservation (\ref{independent}), and has a unique $f_{\pm}\in X$, holding all assertions in Theorem \ref{scattering}. Furthermore, map $f_{0}\mapsto f_{\pm}$ is a continuous bijection, mapping $B_{\delta}(0)$ in $X$ to $B_{2\delta}(0)$.
		\end{thm}
		\begin{rem}
			This theorem builds the asymptotic completeness for the equation (\ref{abstract}) with small initial data.
		\end{rem}
		\vspace{10pt}
		\subsection{Abstract framework-V}
		Notice that, the assumption (V) in the abstract framework-Z requires that the decay power $\alpha$ equals to $\frac{2}{p+1}$, which, in some extent, makes the framework not very flexible for application. Next, we introduce another abstract framework in \cite{4}, which can finally help us establish the scattering theory for DNLW.
		
		We first introduce V-norm as follows.
		\begin{equation*}
			\|u\|_{V}:=\sup_{t\in\mathbb{R}}(1+|t|)^{\alpha}|u(t)|_{3},
		\end{equation*}
		with the corresponding space $V:=\lbrace u\in C(\mathbb{R};X_{3})\big| \|u\|_{V}<\infty\rbrace$.
		
		Then, we can build the following ``V-version" theorems.
		\begin{thm}\label{theorem1}
			We suppose the assumption (I)$\sim$(III), (V)$\sim$(VII), allowing bigger range of $\alpha$ as follows.
			\begin{equation*}
				\frac{1}{p}<\alpha<1.
			\end{equation*}
			Naturally, we replace ``Z" with ``V" in (VII). Then, if $f_{-}\in X$, $U_{0}(\cdot)f_{-}\in V$, there exists $\delta>0$, such that if $\|U_{0}(\cdot)f_{-}\|_{V}<\delta$, all assertions in Theorem \ref{scattering} can also be satisfied, with $u\in V$.
		\end{thm}
		\begin{rem}
			The assumption of the abstract framework-V is basically consistent with the assumption of the abstract framework-Z, but here we highlight several differences. First, we no longer introduce the $Z$ space and allow $\alpha$ to take a larger range $\frac{1}{p}<\alpha<1$ (naturally, the original $\alpha=\frac{2}{p+1}$ belongs to this range). And since this framework is based on $V$ space, the solution $u$ will also have a stronger property: $u\in V$, as $V\subseteq Z$.
		\end{rem}
		\begin{thm}\label{theorem2}
			We suppose all the assumptions in Theorem \ref{theorem1}, without assuming the norm of  $U_{0}(\cdot)f_{-}$ is sufficiently small. Then there exists $T>-\infty$, and a unique solution $u\in C(I;X\cap X_{3})$, satisfying the equation ($\ast_{-\infty},f_{-}$), 
			\begin{equation*}
				\sup_{-\infty<t\le T}(1+|t|)^{\alpha}|u(t)|_{3}<\infty,
			\end{equation*}
			and (\ref{asy}), (\ref{f}).
		\end{thm}
		\begin{rem}
			This theorem is just the V-version of Theorem \ref{theorem3}, establishing the well-posedness on $(-\infty,T]$, with general data $f_{-}\in X$.
		\end{rem}
		We also have V-version of Theorem \ref{wee}, which ensures the asymptotic completeness with small initial data.
		\begin{thm}\label{theorem3.6}
			We suppose (I)$\sim$(III), (V)$\sim$(VII), allowing bigger range of $\alpha$ as follows.
			\begin{equation*}
				\frac{1}{p}<\alpha<1.
			\end{equation*}
			Naturally, we replace ``Z" with ``V" in (VII). If $f_{0}\in X$, $U_{0}(\cdot)f_{0}\in V$, there exists $\delta>0$, such that if $\|U_{0}(\cdot)f_{0}\|_{V}<\delta$, we have a unique strong solution $u\in C(\mathbb{R};X)\cap V$ of the equation (\ref{abstract}), with the initial data $f_{0}$. Furthermore, we have a unique $f_{\pm}\in X$, satisfying all assertions in Theorem \ref{scattering}.
		\end{thm}
		
		\section{Proof of Theorem \ref{main}, \ref{main2}, \ref{main3}}
		In order to use the abstract framework-V, we shall first introduce the discrete Sobolev spaces and their multiplier theorem, which can be referred to \cite{24}. For convenience, we still denote the discrete Sobolev spaces with notations in the usual Sobolev spaces.
		
		\subsection{Discrete Sobolev space}
		\begin{defi}\label{3.3}
			For $f:\mathbb{Z}^{d}\to\mathbb{C}$, $s\in \mathbb{R}$, $1\le p\le \infty$, we define the homogeneous Soobolev norm and the inhomogeneous Sobolev norm as follows.
			\begin{equation*}
				\|f\|_{\dot{W}^{s,p}(\mathbb{Z}^{d})}:=\|\mathcal{F}^{-1}(\omega(\xi)^{s}\mathcal{F}(f))\|_{\ell^{p}(\mathbb{Z}^{d})},
			\end{equation*}
			\begin{equation*}
				\|f\|_{W^{s,p}(\mathbb{Z}^{d})}:=\|\mathcal{F}^{-1}(\langle\omega(\xi)\rangle^{s}\mathcal{F}(f))\|_{\ell^{p}(\mathbb{Z}^{d})},
			\end{equation*}
			where $\mathcal{F}$, $\mathcal{F}^{-1}$ are discrete Fourier transform and its inverse. Then, the corresponding discrete Sobolev spaces are defined as follows.
			\begin{equation*}
				\dot{W}^{s,p}(\mathbb{Z}^{d}):=\lbrace f:\mathbb{Z}^{d}\to\mathbb{C}\big|\|f\|_{\dot{W}^{s,p}(\mathbb{Z}^{d})}<\infty \rbrace,
			\end{equation*}
			\begin{equation*}
				W^{s,p}(\mathbb{Z}^{d}):=\lbrace f:\mathbb{Z}^{d}\to\mathbb{C}\big|\|f\|_{W^{s,p}(\mathbb{Z}^{d})}<\infty \rbrace.
			\end{equation*}
			Conventionally, we also denote $\dot{W}^{s,2}(\mathbb{Z}^{d})$, $W^{s,2}(\mathbb{Z}^{d})$ as $\dot{H}^{s}(\mathbb{Z}^{d})$, $H^{s}(\mathbb{Z}^{d})$, respectively.
		\end{defi}
		\begin{rem}
			In fact, the definitions of the discrete Sobolev spaces in \cite{24} is different from the definitions above. Specifically, the definitions in \cite{24} are
			\begin{equation*}
				\|f\|_{\dot{W}^{s,p}(\mathbb{Z}^{d})}:=\|\mathcal{F}^{-1}(|\xi|^{s}\mathcal{F}(f))\|_{\ell^{p}(\mathbb{Z}^{d})}, \quad \|f\|_{W^{s,p}(\mathbb{Z}^{d})}:=\|\mathcal{F}^{-1}(\langle\xi\rangle^{s}\mathcal{F}(f))\|_{\ell^{p}(\mathbb{Z}^{d})}.
			\end{equation*}
			Since we all consider the lattices $\mathbb{Z}^{d}$ in this paper, we can see that the above two definitions of the discrete Sobolev spaces are equivalent, based on the following multiplier theorems. The reason why we choose the former definitions is that it can be compatible with the energy conservation of DW.
		\end{rem}
	   \begin{thm}\label{muti}
			Suppose the Fourier multiplier of the operator $T_{m}$ is $m(\xi)$, i.e. $\mathcal{F}(T_{m}f)(\xi)=m(\xi)\mathcal{F}(f)(\xi)$, $\forall \xi\in \mathbb{T}^{d}$. If the multiplier $m(\xi)$ satisfies
			\begin{equation*}
				\left|\partial^{\beta}m(\xi)\right|\le C_{\beta}|\xi|^{-|\beta|}, \quad \forall \xi\in\mathbb{T}^{d}-\lbrace0\rbrace,
			\end{equation*}
			where $\beta$ is any multi-index that $|\beta|\le d+2$. Then for $1<p<\infty$, there exists a constant $C_{p}>0$, such that
			\begin{equation*}
				\|T_{m}f\|_{\ell^{p}(\mathbb{Z}^{d})}\le C_{p}\|f\|_{\ell^{p}(\mathbb{Z}^{d})}.
			\end{equation*}
		\end{thm}
		\vspace{10pt}
		\subsection{Proof of major theorems}
		Corresponding to the notations in the abstract framework-V, we take $X:=\dot{H}^{1}(\mathbb{Z}^{d})\times \ell^{2}(\mathbb{Z}^{d})$, $X_{3}:=\ell^{p+1}(\mathbb{Z}^{d})\times\ell^{p+1}(\mathbb{Z}^{d})$, $X_{1}:=\lbrace0\rbrace\times\ell^{1+\frac{1}{p}}(\mathbb{Z}^{d})$, $X_{4}:=\dot{W}^{1, p+1}(\mathbb{Z}^{d})\times\ell^{p+1}(\mathbb{Z}^{d})$, and $\alpha=\beta_{d}(1-\frac{4}{p+1})$. Besides, we take the functional $G$ on $X_{3}$ as follows.
		\begin{equation}\label{functional}
			G(\left[f,g\right]):=\frac{-\mu}{p+1}\|f\|_{\ell^{p+1}(\mathbb{Z}^{d})}^{p+1}, \quad \forall \left[f,g\right]\in X_{3}.
		\end{equation}
		
		Notice that our desired Theorem \ref{main}, \ref{main2}, \ref{main3}, correspond to Theorem \ref{theorem1}, \ref{theorem2}, \ref{theorem3.6} in the abstract framework-V, respectively. Next, we just need to verify all the assumptions imposed. 
		\begin{proof}[Proof of the assumption (I)]
			From \cite{25}, we know the energy conservation:
			\begin{equation*}
				E(\widetilde{u})(t):=\frac{1}{2}\|u\|_{\dot{H}^{1}(\mathbb{Z}^{d})}^{2}+\frac{1}{2}\|\partial_{t}u\|_{\ell^{2}(\mathbb{Z}^{d})}^{2}\; \text{is independent with $t$}.
			\end{equation*}
			This energy is just the norm in $X=\dot{H}^{1}(\mathbb{Z}^{d})\times\ell^{2}(\mathbb{Z}^{d})$, so $\lbrace U_{0}(t)\rbrace_{t\in \mathbb{R}}$ is an one parameter unitary operator group on $X$.
		\end{proof}
		\begin{proof}[Proof of the assumption (II)]
		    For $F=(f_{1},f_{2}), G=(g_{1},g_{2})$, the Lipschitz-type control
			\begin{equation*}
				|PF-PG|_{1}\lesssim (|F|_{3}+|G|_{3})^{p-1}|F-G|_{3},
			\end{equation*}
			can be reduced to 
			\begin{equation*}
				\||f_{1}|^{p-1}f_{1}-|g_{1}|^{p-1}g_{1}\|_{\ell^{1+\frac{1}{p}}(\mathbb{Z}^{d})}\lesssim (\|f_{1}\|_{\ell^{p+1}(\mathbb{Z}^{d})}+\|g_{1}\|_{\ell^{p+1}(\mathbb{Z}^{d})})^{p-1}\|f_{1}-g_{1}\|_{\ell^{p+1}(\mathbb{Z}^{d})}.
			\end{equation*}
			It's a direct consequence of the H\"{o}lder inequality and 
			\begin{equation*}
				\Big||f_{1}|^{p-1}f_{1}-|g_{1}|^{p-1}g_{1}\Big|\lesssim (|f_{1}|^{p-1}+|g_{1}|^{p-1})|f_{1}-g_{1}|.
			\end{equation*}
		\end{proof}
		\begin{proof}[Proof of the assumption (III)]
			For $1\le p<q\le\infty$, we have the embedding $\ell^{p}(\mathbb{Z}^{d})\subseteq\ell^{q}(\mathbb{Z}^{d})$. Then we have $X_{1},X\subseteq X_{4}$. Based on Theorem \ref{muti}, we see that $|\xi|$ satisfies the Hormander-Mikhlin condition, which leads to $\|\cdot\|_{\dot{W}^{1, p+1}(\mathbb{Z}^{d})}\lesssim \|\cdot\|_{\ell^{p+1}(\mathbb{Z}^{d})}$, and $X_{3}\subseteq X_{4}$. The density is obvious.
		\end{proof}
		\begin{proof}[Proof of the assumption (V)]
			From Theorem \ref{jj} and Remark \ref{d}, we can get 
			\begin{equation*}
				|U_{0}(t)F|_{3}\lesssim (1+|t|)^{-\beta_{d}(1-\frac{4}{p+1})}|F|_{1}.
			\end{equation*}
			Simple calculation shows that when $p>p_{d}$, we have $\beta_{d}(1-\frac{4}{p+1})>\frac{1}{p}$, which satisfies the requirement.
			
			Next we show that  $U_{0}(t)$ can be continuously extended to $X_{3}$, with value $X_{4}$. Suppose $F=(f_{1},f_{2})\in X_{3}$, we have the following expression from \cite{25}.
			
			\[U_{0}(t)F=
			\begin{bmatrix}
				\cos(t\sqrt{-\Delta}) & \dfrac{\sin(t\sqrt{-\Delta})}{\sqrt{-\Delta}} \\
				-\sin(t\sqrt{-\Delta})\cdot \sqrt{-\Delta} & \cos(t\sqrt{-\Delta})
			\end{bmatrix} 
			\begin{bmatrix}
				f_{1} \\
				f_{2}
			\end{bmatrix},\]  
			where $\cos(t\sqrt{-\Delta})f:=\mathcal{F}^{-1}(\cos(t\cdot\omega(\xi))\mathcal{F}(f))$, and other symbols are defined similarly.
			
			From the definition, we deduce that
			\begin{equation*}
				|U_{0}(t)F|_{4}\lesssim \|\cos(t\sqrt{-\Delta})f_{1}\|_{\dot{W}^{1,p+1}(\mathbb{Z}^{d})}+\left\|\frac{\sin(t\sqrt{-\Delta})}{\sqrt{-\Delta}}f_{2}\right\|_{\dot{W}^{1,p+1}(\mathbb{Z}^{d})}
			\end{equation*}
			\begin{equation*}
				+\|\sin(t\sqrt{-\Delta})\cdot \sqrt{-\Delta}f_{1}\|_{\ell^{p+1}(\mathbb{Z}^{d})}+\|\cos(t\sqrt{-\Delta})f_{2}\|_{\ell^{p+1}(\mathbb{Z}^{d})}
			\end{equation*}
			\begin{equation*}
				:=I_{1}+I_{2}+I_{3}+I_{4}.
			\end{equation*}
			Notice that, the Fourier multipliers of $I_{1}, I_{2}, I_{3}, I_{4}$ are 
			\begin{equation*}
				I_{1}\leftrightarrow \cos(t\cdot\omega(\xi))\omega(\xi), \quad I_{2}\leftrightarrow\sin(t\cdot\omega(\xi)),
			\end{equation*}
			\begin{equation*}
				I_{3}\leftrightarrow \sin(t\cdot\omega(\xi))\omega(\xi), \quad I_{4}\leftrightarrow \cos(t\cdot\omega(\xi)).
			\end{equation*}
			Simple calculation shows that they satisfy the Hormander-Mikhlin condition, so we have
			\begin{equation*}
				|U_{0}(t)F|_{4}\le C(t) |F|_{3},
			\end{equation*}
			where $C(t)$ is locally bounded in $\mathbb{R}$. As $X\cap X_{1}$ is dense in $X_{1}$, there follows the $U_{0}(t)U_{0}(s)F=U_{0}(t+s)F,\forall F\in X_{1}$.
		\end{proof}
		\begin{proof}[Proof of assumption (VI)]
			From the definition (\ref{functional}), $G$ is actually defined on the whole $X_{3}$ and continuous.
		\end{proof}
		\begin{proof}[Proof of assumption (VII)]
			For convenience, we suppose $s=0$. To solve the equation ($\ast_{0},F$), we consider the following iteration.
			\begin{equation*}
				\widetilde{u}_{m+1}(t)=U_{0}(t)F+\int_{0}^{t}U_{0}(t-s)P\widetilde{u}_{m}(s)ds,\quad \widetilde{u}_{m}(t):=\left(\widetilde{u}_{m}^{(1)}(t),\widetilde{u}_{m}^{(2)}(t)\right),
			\end{equation*}
			where we take $\widetilde{u}_{0}(t)=(0,0)$. Now we claim that when $T\ll1$, we have $\sup_{t\in [-T,T] }|\widetilde{u}_{m}(t)|_{2}\le C,\forall m\in \mathbb{N}$, where $C$ only depends on $|F|_{2}$ (we can take $C=2|F|_{2}$). When $m=0$, the conclusion is obvious. Then, using the induction, we suppose the case $m=k$ holds and prove the case $m=k+1$. From the iteration and the Minkowski inequality, we see
			\begin{equation*}
				|\widetilde{u}_{k+1}(t)|_{2}\le |F|_{2}+\int_{0}^{t}|P\widetilde{u}_{k}(s)|_{2}ds= |F|_{2}+\int_{0}^{t}\|\widetilde{u}_{k}^{(1)}(s)\|_{\ell^{2p}(\mathbb{Z}^{d})}^{p}ds
			\end{equation*}
			\begin{equation*}
				\le |F|_{2}+\int_{0}^{t}\|\widetilde{u}_{k}^{(1)}(s)\|_{\ell^{2}(\mathbb{Z}^{d})}^{p}ds\le \frac{C}{2}+|t|C^{p}.
			\end{equation*}
			Then we take $T\ll_{C}1$, and get $|\widetilde{u}_{k+1}(t)|_{2}\le C$, which completes  the induction.
						
			Next, we can use the contraction mapping to find the solution of the equation ($\ast_{0},F$). Specifically, from the Minkowski inequality and the H\"{o}lder inequality, we have
			\begin{equation*}
				\|\widetilde{u}_{m+1}-\widetilde{u}_{m}\|_{C([-T,T];X)}=\sup_{t\in [-T,T]}\left|\int_{0}^{t}U_{0}(t-s)\left(P\widetilde{u}_{m}(s)-P\widetilde{u}_{m-1}(s)\right)ds\right|_{2}\le \sup_{t\in [-T,T]}\int_{0}^{t}|P\widetilde{u}_{m}(s)-P\widetilde{u}_{m-1}(s)|_{2}ds
			\end{equation*}
			\begin{equation*}
				=\sup_{t\in [-T,T]}\int_{0}^{t}\left\||\widetilde{u}_{m}^{(1)}(s)|^{p-1}\widetilde{u}_{m}^{(1)}(s)-|\widetilde{u}_{m-1}^{(1)}(s)|^{p-1}\widetilde{u}_{m-1}^{(1)}(s)\right\|_{\ell^{2}(\mathbb{Z}^{d})}ds
			\end{equation*}
			\begin{equation*}
				\lesssim \sup_{t\in [-T,T]}\int_{0}^{t}\left\|\left(|\widetilde{u}_{m}^{(1)}(s)|^{p-1}+|\widetilde{u}_{m-1}^{(1)}(s)|^{p-1}\right)\left(\widetilde{u}_{m}^{(1)}(s)-\widetilde{u}_{m-1}^{(1)}(s)\right)\right\|_{\ell^{2}(\mathbb{Z}^{d})}ds
			\end{equation*}
			\begin{equation*}
				\lesssim \sup_{t\in [-T,T]}\int_{0}^{t}\left(\left\|\widetilde{u}_{m}^{(1)}(s)\right\|_{\ell^{2p}(\mathbb{Z}^{d})}^{p-1}+\left\|\widetilde{u}_{m-1}^{(1)}(s)\right\|_{\ell^{2p}(\mathbb{Z}^{d})}^{p-1}\right)\|\widetilde{u}_{m}^{(1)}(s)-\widetilde{u}_{m-1}^{(1)}(s)\|_{\ell^{2p}(\mathbb{Z}^{d})}ds
			\end{equation*}
			\begin{equation*}
				\lesssim C^{p-1}T\cdot \|\widetilde{u}_{m}-\widetilde{u}_{m-1}\|_{C([-T,T];X)}.
			\end{equation*}
			Thus, when $T\ll1$, we have the contraction $\|\widetilde{u}_{m+1}-\widetilde{u}_{m}\|_{C([-T,T];X)}\le \frac{1}{2}	\|\widetilde{u}_{m}-\widetilde{u}_{m-1}\|_{C([-T,T];X)}$. Directly, $\lbrace\widetilde{u}_{m}\rbrace_{m=0}^{\infty}$ is a Cauchy sequence in $C([-T,T];X)$, with limit $\widetilde{u}\in C([-T,T];X)$ as the locally well-posed strong solution of the equation ($\ast_{0},F$).
			
			As the time of existence only relies on the norm of the initial data, it suffices to ensure that $|\widetilde{u}(t)|_{2}$ will never blow up in finite time. From \cite{25}, we also have the following energy conservation:
			\begin{equation*}
				\frac{1}{2}|\widetilde{u}(t)|_{2}^{2}-\frac{\mu}{p+1}\|u(t)\|_{\ell^{p+1}(\mathbb{Z}^{d})}^{p+1} \text{ is independent with $t$.}
			\end{equation*}
			Since $\widetilde{u}\in V$, $\|\widetilde{u}\|_{V}\ll 1$, we see that $\sup_{t\in I}\|u(t)\|_{\ell^{p+1}(\mathbb{Z}^{d})}\ll1$. Combining with the energy conservation, there is no any blow up for $\widetilde{u}$ and we can extend it to the whole $I$.
		\end{proof}
		\begin{rem}\label{d}
			Except from the assumptions (I)$\sim$(VII), we also need to ensure that $U_{0}(\cdot)F_{-},U_{0}(\cdot)F_{0}\in V$, with sufficiently small norms. Then it is reduced to the following estimate.
			\begin{equation*}
				|U_{0}(\cdot)F|_{3}=\left|\left[u,\partial_{t}u\right]\right|_{3}=\|u\|_{\ell^{p+1}(\mathbb{Z}^{{d}})}+\|\partial_{t}u\|_{\ell^{p+1}(\mathbb{Z}^{{d}})}
			\end{equation*}
			\begin{equation*}
				\lesssim (1+|t|)^{-\beta_{d}(1-\frac{4}{p+1})}\|F\|_{\ell^{1+\frac{1}{p}}(\mathbb{Z}^{d})},
			\end{equation*}
			where $F=(f,g)\in Y=\ell^{1+\frac{1}{p}}(\mathbb{Z}^{d})\times\ell^{1+\frac{1}{p}}(\mathbb{Z}^{d})$, $u$ is the solution of DW with the initial data. Recalling Remark
			\ref{rem}, the above estimate is a direct consequence of Theorem \ref{jj}.
		\end{rem}
		\begin{rem}
			Actually, the proof of the assumption (VII) can be abstracted as the following framework in \cite{26}:
			\begin{itemize}
				\item For $T=[t_{0},+\infty)$ or $(-\infty,+\infty)$, $B$ is a Banach space, and the propagator $W(t,s)$ is a $(t,s)$-continuous bounded linear operator for $t,s\in T, t\ge s$, satisfying 
				\begin{equation*}
					W(t_{3},t_{2})W(t_{2},t_{1})=W(t_{3},t_{1}), \quad W(t,t)=Id, 
				\end{equation*}
				where $\forall t,t_{1},t_{2},t_{3}\in T$, $t_{1}\le t_{2}\le t_{3}$.
				\item For $T=[t_{0},+\infty)$, $W$ as above, if $t\in T$, $K_{t}$ is an operator on $B$, not necessarily linear, and has a uniformly local Lipschitz-type control, i.e.
				\begin{equation*}
					\left\|K_{t}(u)-K_{t}(v)\right\|_{B}\le C(M,t^{\ast})\|u-v\|_{B},
				\end{equation*}
				where $t\le t^{\ast}$, $\|u\|_{B}, \|v\|_{B}\le M$, $C(M,t^{\ast})$ only depends on $M,t^{\ast}$. Furthermore, if $K_{t}(u)$ is continuous in $T\times B\ni (t,u)$, then for any initial data $u_{0}\in B$, the following equation
				\begin{equation*}
					u(t)=W(t,t_{0})u_{0}+\int_{t_{0}}^{t}W(t,s)K_{s}(u(s))ds,
				\end{equation*}
				has a unique local solution $u\in C(I;B)$, $t_{0}\in I\subseteq T$. Moreover, we have the continuation principle: if $t^{\ast}$ is the maximal time of existence and $t^{\ast}<\infty$, then $\|u(t)\|_{B}\to +\infty$, $t\to t^{\ast}$.
			\end{itemize}
		\end{rem}
		\begin{rem}\label{4.4}
			In this paper, we only consider the case $d=2,3,4,5$, instead of higher dimensions and $1$ dimension. For higher dimensions, the degenerate points of the phase function will be more complex, and there is no sufficient harmonic analysis tool to study it. For $1$ dimension, we also fail to establish the scattering theory as integrand $\frac{1}{\omega(\xi)}$ has greater singularity at $0$, leading to the loss of the decay rate in the integral (\ref{Green}) or (\ref{Oscillatory}). This fact can also be reflected by the estimate in \cite{1}. Actually, \cite{1} shows that the integral (\ref{d-1}) has the decay rate $|t|^{-(d-1)}$, so we have no decay in $1$ dimension. Next, we strictly prove that there has no uniform decay estimate, which suffices to show 
			\begin{equation}\label{kl}
				\liminf_{|t|\to +\infty}\sup_{x\in \mathbb{Z}}|G(x,t)|=
				\liminf_{|t|\to +\infty}\sup_{x\in \mathbb{Z}}\left|\frac{1}{2\pi}\int_{\mathbb{T}}e^{ix\cdot\xi}\frac{\sin(t\omega(\xi))}{\omega(\xi)}d\xi\right|\ge\frac{1}{2}.
			\end{equation}
			We also decompose (\ref{part}) into two parts:
			\begin{equation*}
				G(x,t)=\frac{1}{2\pi}\int_{-\delta}^{\delta}e^{ix\cdot\xi}\frac{\sin(t\omega(\xi))}{\omega(\xi)}d\xi+\frac{1}{2\pi}\int_{\mathbb{T}-(-\delta,\delta))}e^{ix\cdot\xi}\frac{\sin(t\omega(\xi))}{\omega(\xi)}d\xi:=I_{1}(x,t)+I_{2}(x,t),
			\end{equation*}
			where $0<\delta\ll 1$. For $I_{2}$, we can rewrite it as the form of (\ref{Oscillatory}), with the phase function $\phi(v,\xi)$. Then $\phi(v,\xi)=v\cdot\xi-2|\sin(\frac{\xi}{2})|$ has no degenerate points, which leads to 
			\begin{equation*}
				\sup_{x\in \mathbb{Z}}I_{2}(x,t)\lesssim (1+|t|)^{-\frac{1}{2}}.
			\end{equation*}
			Taking $x=0$, we deduce that
			\begin{equation*}
				\liminf_{|t|\to +\infty}\sup_{x\in \mathbb{Z}}|G(x,t)|=\liminf_{|t|\to +\infty}\sup_{x\in \mathbb{Z}}|I_{1}(x,t)|\ge \liminf_{|t|\to +\infty}\left|\frac{1}{2\pi}\int_{-\delta}^{\delta}\frac{\sin(t\omega(\xi))}{\omega(\xi)}d\xi\right|.
			\end{equation*}
			We can also consider the frequently-used change of variables $\zeta:=2\sin(\frac{\xi}{2})$, and derive
			\begin{equation*}
				\frac{1}{2\pi}\int_{-\delta}^{\delta}\frac{\sin(t\omega(\xi))}{\omega(\xi)}d\xi=\frac{1}{2\pi}\int_{-\epsilon}^{\epsilon}\frac{\sin(t\zeta)}{\zeta}\cdot\frac{1}{\cos(\frac{\xi(\zeta)}{2})}d\zeta, \quad \epsilon=2\sin\left(\frac{\delta}{2}\right), \; \xi(\zeta)=2\arcsin\left(\frac{\zeta}{2}\right).
			\end{equation*}
			Next, we take $\eta:=t\zeta$, $2N\pi < t\epsilon \le (2N+1)\pi$, $N\in \mathbb{Z}^{+}$, and decompose the integral again as follows.
			\begin{equation}\label{3.8}
				\frac{1}{2\pi}\int_{-t\epsilon}^{t\epsilon}\frac{\sin(\eta)}{\eta}\cdot\frac{1}{\cos(\arcsin(\frac{\eta}{2t}))}d\eta=\frac{1}{2\pi}\int_{-2N\pi }^{2N\pi }\;\;+\;\;\frac{1}{2\pi}\int_{2N\pi }^{t\epsilon}\;\;+\;\;\frac{1}{2\pi}\int_{-t\epsilon}^{-2N\pi}
			\end{equation}
			\begin{equation*}
				:=J_{1}+J_{2}+J_{3}.
			\end{equation*}
			Notice that, the denominator $\eta\cdot\cos(\arcsin(\frac{\eta}{2t}))$ is an increasing function of $\eta$, then we have
			\begin{equation*}
				|J_{2}|\le \frac{1}{2\pi}\int_{2N\pi }^{(2N+1)\pi }\frac{\sin(\eta)}{\eta}\cdot\frac{1}{\cos(\arcsin(\frac{\eta}{2t}))}d\eta\lesssim \frac{1}{N}.
			\end{equation*}
			Similarly, $|J_{3}|\lesssim \frac{1}{N}$. Then letting $t\to +\infty$, $N\to +\infty$, we can derive that
			\begin{equation*}
				\frac{1}{2\pi}\int_{-\delta}^{\delta}\frac{\sin(t\omega(\xi))}{\omega(\xi)}d\xi\longrightarrow\frac{1}{2\pi}\int_{-\infty}^{+\infty}\frac{\sin(\eta)}{\eta}d\eta=\frac{1}{2},\quad t\to +\infty.
			\end{equation*}
			Thus, we have proved the absence of decay in $1$ dimension.
		\end{rem}
		\begin{rem}
			Although there has no decay for general $1$ dimensional DW, we do have some decay estimate for special cases. To be more specific, we can consider the following DW with vanishing velocity. 
			\begin{equation}\label{112}
				\left\{
				\begin{aligned}
					& \partial_{t}^{2} u(x,t)-\Delta u(x,t) =0,  \\
					& u(x,0)=f(x), \partial_{t}u(x,0)=0,\quad (x,t)\in \mathbb{Z}^d\times \mathbb{R}.
				\end{aligned}
				\right.
			\end{equation}
			Notice that the solution $u$ of (\ref{112}) can be expressed as $u=H(\cdot,t)\ast f$, where
			\begin{equation*}
				H(x,t):=\frac{1}{2\pi}\int_{\mathbb{T}}e^{ix\cdot\xi}\cos(2t\sin(\frac{\xi}{2}))d\xi.
			\end{equation*}
			Changing the variable $\theta:=\frac{\xi}{2}$, we see 
			\begin{equation}\label{Bessel}
				H(x,t)=\frac{1}{2\pi}\int_{-\pi}^{\pi}e^{-2it(p\theta-\sin(\theta))}d\theta:=\frac{1}{2\pi}\int_{-\pi}^{\pi}e^{-2it\phi(p,\theta)}d\theta, \quad x=pt.
			\end{equation}
			Simple calculation shows that degenerate points of the phase function $\phi(p,\theta)$ are $\theta=0,\pm\pi$, and $\frac{\partial^{3}}{\partial \theta^{3}}\phi(p,\theta)\Big|_{\theta=0,\pm\pi}=1>0$. Using the Van der Corput lemma, we can get
			\begin{equation}\label{-1/3}
				\|H(\cdot,t)\|_{\ell^{\infty}(\mathbb{Z})}=\sup_{x\in \mathbb{Z}}|H(x,t)|\lesssim (1+|t|)^{-\frac{1}{3}}.
			\end{equation}	
			On the other hands, we obviously have
			\begin{equation*}
				\|H(\cdot,t)\|_{\ell^{2}(\mathbb{Z})}=\frac{1}{2\pi}\int_{\mathbb{T}}\cos^{2}(2t\sin(\frac{\xi}{2}))d\xi\lesssim 1.
			\end{equation*}
			Then we can directly get 
			\begin{equation*}
				\|u(\cdot,t)\|_{\ell^{r}(\mathbb{Z})}\lesssim (1+|t|)^{-\frac{1}{3}(1-\frac{2}{k})}\|f\|_{\ell^{q}(\mathbb{Z})}, \quad 1+\frac{1}{r}=\frac{1}{k}+\frac{1}{q},
			\end{equation*}
			where $k\ge2$, $1\le r,q\le \infty$.
			
			It is worth to mention that $H(x,t)=J_{2|x|}(2t)$, where $J_{v}(t)$ is the Bessel function. Classically, we have the following asymptotic \cite{20}
				\begin{equation}\label{89}
					J_{v}(t)=\sqrt{\frac{2}{\pi t}}\cos(t-\frac{\pi v}{2}-\frac{\pi}{4})+R_{v}(t),\quad Re(v)>-\frac{1}{2}, t\ge1,
				\end{equation}
				where the remainder $R_{v}$ satisfies $|R_{v}(t)|\le C_{v}t^{-\frac{3}{2}}$, $C_{v}$ depends on $v$.
			
			Although $H(x,t)$ has decay rate $t^{-\frac{1}{2}}$, we can't refine the power $-\frac{1}{3}$ in the uniform decay estimate (\ref{-1/3}) to $-\frac{1}{2}$. Since the constant $C_{v}$ in the asymptotic (\ref{89}) is not uniform in $v$.
		\end{rem}
		\section{Proof of Theorem \ref{S1}, \ref{S1.5}, \ref{S2}, \ref{S3}}
		The Strichartz estimate is a very important tool in the dispersive equations, which also has a deep connection with the restriction \cite{29,30}. In this section, we will first prove the Strichartz estimate of DW, and establish the scattering theory again.
		
		We shall mention that the ideas here are similar with the ideas in \cite{18}, where the author used to establish the scattering theory of the discrete nonlinear Schr\"{o}dinger equations. However, there still have some differences with \cite{18}, due to the complexity of the wave equation.
		
		\subsection{The Strichartz estimate of DW}
		Considering the following inhomogeneous DW:
		\begin{equation}\label{pi}
			\left\{
			\begin{aligned}
				& \partial_{t}^{2} u(x,t)-\Delta u(x,t) =F(x,t),  \\
				& u(x,0) = f(x), \partial_{t}u(x,0)=g(x),\quad (x,t)\in \mathbb{Z}^d\times \mathbb{R}.
			\end{aligned}
			\right.
		\end{equation}
		From the classical results in \cite{28}, we can derive the following Strichartz estimate
		\begin{thm}\label{Strichartz}
			For any $\beta_{d}$-admissible pairs $(q,r)$, $(\widetilde{q}, \widetilde{r})$, we have
			\begin{equation}\label{11}
				\|e^{it\sqrt{-\Delta}}v\|_{L_{t}^{q}\ell^{r}_{x}(\mathbb{R}\times \mathbb{Z}^{d})}\lesssim \|v\|_{\ell^{2}(\mathbb{Z}^{d})}, \quad \forall t\in \mathbb{R},
			\end{equation}
			\begin{equation}\label{12}
				\left\|\int_{0}^{t}e^{i(t-s)\sqrt{-\Delta}}G(s,\cdot)ds\right\|_{L_{t}^{q}\ell^{r}_{x}(\mathbb{R}\times\mathbb{Z}^{d})}\lesssim \|G\|_{L_{t}^{\widetilde{q}^{'}}\ell_{x}^{\widetilde{r}^{'}}(\mathbb{R}\times\mathbb{Z}^{d})}.
			\end{equation}
			In particular, if $u$ is the solution of the equation $(\ref{pi})$, then we can deduce that
			\begin{equation}\label{12.5}
				\|u\|_{L_{t}^{q}\ell_{x}^{r}(\mathbb{R}\times \mathbb{Z}^{d})}\lesssim \|f\|_{\ell^{2}(\mathbb{Z}^{d})}+\|g\|_{\ell^{\frac{2d}{d+2}}(\mathbb{Z}^{d})}+\Big\|F\Big\|_{L_{t}^{\widetilde{q}^{'}}\ell_{x}^{\frac{d\widetilde{r}^{'}}{d+\widetilde{r}^{'}}}(\mathbb{R}\times\mathbb{Z}^{d})}
			\end{equation}
			\begin{equation}\label{13}
				\|u\|_{L_{t}^{q}\dot{W}_{x}^{1,r}(\mathbb{R}\times \mathbb{Z}^{d})}\lesssim \|f\|_{\dot{H}^{1}(\mathbb{Z}^{d})}+\|g\|_{\ell^{2}(\mathbb{Z}^{d})}+\|F\|_{L_{t}^{\widetilde{q}^{'}}\ell_{x}^{\widetilde{r}^{'}}(\mathbb{R}\times\mathbb{Z}^{d})},
			\end{equation}
			\begin{equation}\label{14}
				\|\partial_{t}u\|_{L_{t}^{q}\ell_{x}^{r}(\mathbb{R}\times \mathbb{Z}^{d})}\lesssim \|f\|_{\dot{H}^{1}(\mathbb{Z}^{d})}+\|g\|_{\ell^{2}(\mathbb{Z}^{d})}+\|F\|_{L_{t}^{\widetilde{q}^{'}}\ell_{x}^{\widetilde{r}^{'}}(\mathbb{R}\times\mathbb{Z}^{d})}.
			\end{equation}
		\end{thm}
		\begin{proof}
		    Let $V(t):=e^{it\sqrt{-\Delta}}$, we can easily verify the following assumptions in \cite{28}, from the uniform decay estimate in Section 2.
			\begin{itemize}
				\item $\|V(t)f\|_{\ell^{2}(\mathbb{Z}^{d})}\lesssim \|f\|_{\ell^{2}(\mathbb{Z}^{d})}$, \quad $\forall t \in \mathbb{R}$.
				\item $\|V(s)\left(V(t)\right)^{\ast}f\|_{\ell^{\infty}(\mathbb{Z}^{d})}\lesssim (1+|t-s|)^{-\beta_{d}}\|f\|_{\ell^{1}(\mathbb{Z}^{d})}$, \quad $t,s\in \mathbb{R}$.
			\end{itemize}
			Thus, we get the estimates (\ref{11}) and (\ref{12}). To control the solution $u$, we can use the following Duhamel formulas
			\begin{equation*}
				u(t)=\cos(t\sqrt{-\Delta})f+\frac{\sin(t\sqrt{-\Delta})}{\sqrt{-\Delta}}g+\int_{0}^{t}\frac{\sin((t-s)\sqrt{-\Delta})}{\sqrt{-\Delta}}F(s)ds,
			\end{equation*} 
			\begin{equation*}
				\partial_{t}u(t)=-\sin(t\sqrt{-\Delta})\cdot\sqrt{-\Delta}f+\cos(t\sqrt{-\Delta})g+\int_{0}^{t}\cos((t-s)\sqrt{-\Delta})F(s)ds.
			\end{equation*}
			For (\ref{12.5}), we can use the Sobolev embedding for the discrete Sobolev spaces (e.g.\cite{24}) and get
			\begin{equation*}
				\|u\|_{L_{t}^{q}\ell_{x}^{1,r}(\mathbb{R}\times \mathbb{Z}^{d})}\lesssim 
				\|f\|_{\ell^{2}(\mathbb{Z}^{d})}+\|\frac{1}{\sqrt{-\Delta}}g\|_{\ell^{2}(\mathbb{Z}^{d})}+\|\frac{1}{\sqrt{-\Delta}}F\|_{L_{t}^{\widetilde{q}}\ell_{x}^{\widetilde{r}}(\mathbb{R}\times\mathbb{Z}^{d})}
			\end{equation*}
			\begin{equation*}
				\lesssim\|f\|_{\ell^{2}(\mathbb{Z}^{d})}+\|g\|_{\ell^{\frac{2d}{d+2}}(\mathbb{Z}^{d})}+\Big\|F\Big\|_{L_{t}^{\widetilde{q}^{'}}\ell_{x}^{\frac{d\widetilde{r}^{'}}{d+\widetilde{r}^{'}}}(\mathbb{R}\times\mathbb{Z}^{d})}.
			\end{equation*}
			For (\ref{13}), we directly have
			\begin{equation*}
				\|u\|_{L_{t}^{q}\dot{W}_{x}^{1,r}(\mathbb{R}\times \mathbb{Z}^{d})}=\|\sqrt{-\Delta}u\|_{L_{t}^{q}\ell_{x}^{r}(\mathbb{R}\times\mathbb{Z}^{d})}\lesssim  \|f\|_{\dot{H}^{1}(\mathbb{Z}^{d})}+\|g\|_{\ell^{2}(\mathbb{Z}^{d})}+\|F\|_{L_{t}^{\widetilde{q}^{'}}\ell_{x}^{\widetilde{r}^{'}}(\mathbb{R}\times\mathbb{Z}^{d})}.
			\end{equation*}
			Parallelly, we can get (\ref{14}), which has collected all the Strichartz estimate we need.
		\end{proof}
		\begin{rem}
			With tiny modification of the proof above, we can replace $\mathbb{R}$ with any time interval $I\subseteq \mathbb{R}$. Besides, we can use the Strichartz norms in Definition \ref{SS} to unify the estimates (\ref{13}) and (\ref{14}) as follows.
			\begin{equation}\label{Uniform S}
				\|\widetilde{u}\|_{S(\mathbb{R}\times \mathbb{Z}^{d})}\le C_{0}\left(\|\widetilde{u}(0)\|_{X}+\|F\|_{N^{0}(\mathbb{R}\times\mathbb{Z}^{d})}\right), \quad \widetilde{u}=(u,\partial_{t}u).
			\end{equation}
		\end{rem}
		\vspace{10pt}
		\subsection{Proof of Theorem \ref{S1}, \ref{S1.5}}
		\begin{proof}[Proof of Theorem \ref{S1}]
			
			In order to find the solution $\widetilde{u}$ for some given $F_{-}$, we consider the integral equation ($\ast_{-\infty},F_{-}$) on $(-\infty,-T]$, where $T\gg1 $. Then we can use the continuation principle, and extend it to $\mathbb{R}$.
			
			As $p\ge 1+\frac{2}{\beta_{d}}+\frac{2}{d}$, it's not very difficult to verify the existence of $0<q_{i}, r_{i}<\infty, i=1,2,3,4,5$, satisfying
			\begin{equation*}
				\frac{1}{q_{4}}+\frac{1}{q_{5}}=\frac{1}{q_{3}'}, \quad \frac{1}{r_{4}}+\frac{1}{r_{5}}=\frac{1}{r_{3}'}+\frac{1}{d},
			\end{equation*} 
			\begin{equation*}
				q_{1}=(p-1)q_{4}, \quad r_{1}=(p-1)r_{4},
			\end{equation*}
			where $(q_{1},r_{1}), (q_{2},r_{2}), (q_{3},r_{3}), (q_{5},r_{5}), (pq_{2}',p\frac{dr_{2}^{'}}{d+r_{2}^{'}})$ are all $\beta_{d}$-admissible. In fact, we can take
			\begin{equation*}
				\frac{1}{q_{1}}=\frac{1}{p-1}\cdot\left(1-\frac{\sigma_{d}}{2}\right),\quad  \frac{1}{r_{1}}=\frac{1}{2}-\frac{1}{(p-1)\sigma_{d}}\cdot\left(1-\frac{\sigma_{d}}{2}\right),
			\end{equation*} 
			\begin{equation*}
				\frac{1}{q_{3}}=\frac{1}{q_{5}}=\frac{\sigma_{d}}{4}, \quad \frac{1}{r_{3}}=\frac{1}{r_{5}}=\frac{1}{4},\quad \frac{1}{q_{2}}=\frac{1}{r_{2}}=\frac{\beta_{d}}{2(\beta_{d}+1)},
			\end{equation*}
			\begin{equation*}
				\frac{1}{q_{4}}=1-\frac{\sigma_{d}}{2},\quad \frac{1}{r_{4}}=\frac{1}{2}+\frac{1}{d},
			\end{equation*}
			where $\sigma_{d}:=\frac{2}{p-1-\frac{2}{d}}$. Simple calculation shows that 
			\begin{equation*}
				p\ge 1+\frac{2}{\beta_{d}}+\frac{2}{d}\Longleftrightarrow \beta_{d}\ge \sigma_{d}.
			\end{equation*}
			Then we introduce the following norm
			\begin{equation*}
				\|V\|_{S_{0}(I\times \mathbb{Z}^{d})}:=\|V\|_{L_{t}^{q_{1}}\ell_{x}^{r_{1}}(I\times\mathbb{Z}^{d})}+\Big\|V\Big\|_{L_{t}^{pq_{2}^{'}}\ell_{x}^{p\frac{dr_{2}^{'}}{d+r_{2}^{'}}}(I\times\mathbb{Z}^{d})},
			\end{equation*}
			where $V=(v_{1},v_{2})$, $\|V\|_{L_{t}^{q}\ell_{x}^{r}(I\times\mathbb{Z}^{d})}:=\|v_{1}\|_{L_{t}^{q}\ell_{x}^{r}(I\times\mathbb{Z}^{d})}+\|v_{2}\|_{L_{t}^{q}\ell_{x}^{r}(I\times\mathbb{Z}^{d})}$.
			
			From the Strichartz estimate, we deduce that 
			\begin{equation*}
				\|U_{0}(t)F_{-}\|_{S_{0}((-\infty,-T]\times \mathbb{Z}^{d})}\lesssim \|U_{0}(t)F_{-}\|_{S^{0}((-\infty,-T]\times \mathbb{Z}^{d})}<+\infty, \quad U_{0}(t):=\begin{bmatrix}
					\cos(t\sqrt{-\Delta}) & \dfrac{\sin(t\sqrt{-\Delta})}{\sqrt{-\Delta}} \\
					-\sin(t\sqrt{-\Delta})\cdot \sqrt{-\Delta} & \cos(t\sqrt{-\Delta})
				\end{bmatrix}.
			\end{equation*}
			Then we can let $T$ sufficiently large, such that $\|U_{0}(t)F_{-}\|_{S_{0}((-\infty,-T]\times \mathbb{Z}^{d})}\le \epsilon$, where $\epsilon>0$ will be determined later. Next, we consider the following iteration.
			\begin{equation*}
				\widetilde{u}_{n+1}(t)=U_{0}(t)F_{-}-\int_{-\infty}^{t}U_{0}(t-s)\begin{bmatrix}
					0\\
					|u_{n}(s)|^{p-1}u_{n}(s)
				\end{bmatrix}ds:=U_{0}(t)F_{-}+DN(\widetilde{u}_{n}),
			\end{equation*}
			where the operators $D,N$ are defined as follows.
			\begin{equation*}
				D:\begin{bmatrix}
					F_{1}\\
					F_{2}
				\end{bmatrix}\mapsto \int_{-\infty}^{t}\begin{bmatrix}
					\cos((t-s)\sqrt{-\Delta}) & \dfrac{\sin((t-s)\sqrt{-\Delta})}{\sqrt{-\Delta}} \\
					-\sin((t-s)\sqrt{-\Delta})\cdot \sqrt{-\Delta} & \cos((t-s)\sqrt{-\Delta})
				\end{bmatrix}
				\begin{bmatrix}
					F_{1}\\
					F_{2}
				\end{bmatrix}ds,
			\end{equation*}
			\begin{equation*}
				N: \begin{bmatrix}
					F_{1}\\
					F_{2}
				\end{bmatrix}\mapsto  
				\begin{bmatrix}
					0\\
					-|F_{1}|^{p-1}F_{1}
				\end{bmatrix}.
			\end{equation*}
			Then in the order $(P_{n})\Rightarrow \widetilde{(P_{n})}\Rightarrow(P_{n+1})$, we can use the induction to prove the following two propositions
			\begin{equation*}
				(P_{n}): \quad \|\widetilde{u}_{n}\|_{S_{0}(\left(-\infty,-T\right]\times\mathbb{Z}^{d})}\le 10\epsilon,
			\end{equation*}
			\begin{equation*}
				\widetilde{(P_{n})}: \quad \|N(\widetilde{u}_{n})\|_{L_{t}^{q_{2}'}\ell_{x}^{r_{2}'}(\left(-\infty,-T\right]\times \mathbb{Z}^{d})}=\||u_{n}|^{p-1}u_{n}\|_{L_{t}^{q_{2}'}\ell_{x}^{r_{2}'}(\left(-\infty,-T\right]\times \mathbb{Z}^{d})}\le \frac{\epsilon}{10C_{0}},
			\end{equation*}
			where $C_{0}$ is the constant in (\ref{Uniform S}).
			
			Using the H\"{o}lder inequality, we can derive that
			\begin{equation*}
				\Big\|N(\widetilde{u}_{n+1})-N(\widetilde{u}_{n})\Big\|_{L_{t}^{q_{3}^{'}}\ell_{x}^\frac{dr_{3}^{'}}{d+r_{3}^{'}}((-\infty,-T]\times \mathbb{Z}^{d})}\lesssim  \|u_{n+1}-u_{n}\|_{L_{t}^{q_{5}}\ell_{x}^{r_{5}}(-\infty,-T]\times \mathbb{Z}^{d})}
			\end{equation*}
			\begin{equation*}
				\times \left(\|u_{n+1}\|_{L_{t}^{q_{1}}\ell_{x}^{r_{1}}((-\infty,-T]\times \mathbb{Z}^{d})}^{p-1}+\|u_{n}\|_{L_{t}^{q_{1}}\ell_{x}^{r_{1}}((-\infty,-T]\times \mathbb{Z}^{d})}^{p-1}\right)\lesssim \epsilon^{p}\|\widetilde{u}_{n+1}-\widetilde{u}_{n}\|_{S^{0}((-\infty,-T]\times \mathbb{Z}^{d}))}.
			\end{equation*}
			Then we can take $\epsilon\ll \frac{1}{C_{0}}$ and ensure that
			\begin{equation*}
				\Big\|N(\widetilde{u}_{n+1})-N(\widetilde{u}_{n})\Big\|_{L_{t}^{q_{3}^{'}}\ell_{x}^\frac{dr_{3}^{'}}{d+r_{3}^{'}}((-\infty,-T]\times \mathbb{Z}^{d})}\le \frac{1}{2C_{0}}\|\widetilde{u}_{n+1}-\widetilde{u}_{n}\|_{S^{0}((-\infty,-T]\times \mathbb{Z}^{d}))}.
			\end{equation*}
			Combing the Strichartz estimate, we can get
			\begin{equation*}
				\|\widetilde{u}_{n+1}-\widetilde{u}_{n}\|_{S^{0}(\left(-\infty,-T\right]\times \mathbb{Z}^{d})}\le \|DN(\widetilde{u}_{n})-DN(\widetilde{u}_{n-1})\|_{S^{0}(\left(-\infty,-T\right]\times \mathbb{Z}^{d})}
			\end{equation*}
			\begin{equation*}
				\le C_{0}\Big\|N(\widetilde{u}_{n})-N(\widetilde{u}_{n-1})\Big\|_{L_{t}^{q_{3}^{'}}\ell_{x}^\frac{dr_{3}^{'}}{d+r_{3}^{'}}((-\infty,-T]\times \mathbb{Z}^{d})}\le \frac{1}{2}\|\widetilde{u}_{n}-\widetilde{u}_{n-1}\|_{S^{0}(\left(-\infty,-T\right]\times \mathbb{Z}^{d})}.
			\end{equation*}
			Thus, $\lbrace\widetilde{u}_{n}\rbrace$ is a Cauchy sequence in $S^{0}(\left(-\infty,-T\right]\times \mathbb{Z}^{d})$, with limit $\widetilde{u}$ as the solution of the equation ($\ast_{-\infty},F_{-}$).
			
			To illustrate the energy conservation (\ref{conservation}), we notice that $\widetilde{u}\in C(\mathbb{R};X)$. Then we can use the same method in \cite{25} to derive
			\begin{equation*}
				\frac{1}{2}\|\widetilde{u}(t)\|_{X}^{2}+\frac{1}{p+1}\|u(t)\|_{\ell^{p+1}(\mathbb{Z}^{d})}^{p+1} \equiv \frac{1}{2}\|\widetilde{u}(0)\|_{X}^{2}+\frac{1}{p+1}\|u(0)\|_{\ell^{p+1}(\mathbb{Z}^{d})}^{p+1}, 	\quad \forall t\in \mathbb{R}.
			\end{equation*}
			From the definition of the Strichartz norm, we see that $u\in L_{t}^{\frac{2(p+1)}{\beta_{d}(p-1)}}\ell_{x}^{p+1}((-\infty,-T]\times\mathbb{Z}^{d})$, then we can take a subsequence $\lbrace t_{j}\rbrace$, such that $\|u(t_{j})\|_{\ell^{P+1}(\mathbb{Z}^{d})}\to 0, t_{j}\to -\infty$, when $j\to +\infty$. Then we have
			\begin{equation*}
				\frac{1}{2}\|\widetilde{u}(t)\|_{X}^{2}+\frac{1}{p+1}\|u(t)\|_{\ell^{p+1}(\mathbb{Z}^{d})}^{p+1}=\lim_{j\to \infty}\frac{1}{2}\|\widetilde{u}(t_{j})\|_{X}^{2}+\frac{1}{p+1}\|u(t_{j})\|_{\ell^{p+1}(\mathbb{Z}^{d})}^{p+1}=\frac{1}{2}\|F_{-}\|_{X}^{2}.
			\end{equation*}
			Similarly, we can derive the same conclusion for $F_{+}$.
		\end{proof}
		\vspace{5pt}
		For the next proof, we just list some important parameters and steps, since it's just a modification of the proof above.
		\begin{proof}[Proof of Theorem \ref{S1.5}]
			Since $p\ge \frac{d(\beta_{d}+2)}{\beta_{d}(d-2)}$, $d\ge 3$, we can find $0<q_{i}, r_{i}<\infty, i=1,2,3,4,5$, satisfying
			\begin{equation*}
				\frac{1}{q_{4}}+\frac{1}{q_{5}}=\frac{1}{q_{3}'}, \quad \frac{1}{r_{4}}+\frac{1}{r_{5}}=\frac{1}{r_{3}'}, 
			\end{equation*} 
			\begin{equation*}
				\frac{1}{q_{4}}=\frac{p-1}{q_{1}}, \quad \frac{1}{r_{4}}+\frac{p-1}{d}=\frac{p-1}{r_{1}},
			\end{equation*}
			where $(q_{1},r_{1}), (q_{2},r_{2}), (q_{3},r_{3}), (q_{5},\frac{dr_{5}}{d+r_{5}}), (pq_{2}',\frac{dpr_{2}^{'}}{d+pr_{2}^{'}})$ are all $\beta_{d}$-admissible. In fact, we can take
			\begin{equation*}
				\frac{1}{q_{1}}=\sigma_{d}\left(\frac{1}{2}-\frac{1}{d}-\frac{1}{2(p-1)}\right),\quad  \frac{1}{r_{1}}=\frac{1}{d}+\frac{1}{2(p-1)},
			\end{equation*} 
			\begin{equation*}
				\frac{1}{q_{3}}=\frac{\sigma_{d}}{8}, \quad \frac{1}{r_{3}}=\frac{3}{8},\quad \frac{1}{q_{2}}=\frac{1}{r_{2}}=\frac{\beta_{d}}{2(\beta_{d}+1)},
			\end{equation*}
			\begin{equation*}
				\frac{1}{q_{5}}=\sigma_{d}\left(\frac{3}{8}-\frac{1}{d}\right), \quad \frac{1}{r_{5}}=\frac{1}{8},\quad \frac{1}{q_{4}}=1-\sigma_{d}\left(\frac{1}{2}-\frac{1}{d}\right), \quad \frac{1}{r_{4}}=\frac{1}{2},
			\end{equation*}
			where $\sigma_{d}:=\frac{1}{p(\frac{1}{2}-\frac{1}{d})-\frac{1}{2}}$. Similarly, we have
			\begin{equation*}
				p\ge \frac{d(\beta_{d}+2)}{\beta_{d}(d-2)}\Longleftrightarrow \beta_{d}\ge \sigma_{d}.
			\end{equation*}
			Then for any time interval $I\in \mathbb{R}$, we introduce the following norm.
			\begin{equation*}
				\|V\|_{S_{0}(I\times\mathbb{Z}^{d})}:=\|v_{1}\|_{L_{t}^{q_{1}}\dot{W}_{x}^{1,r_{1}}(I\times\mathbb{Z}^{d})}+\Big\|v_{1}\Big\|_{L_{t}^{pq_{2}^{'}}\dot{W}_{x}^{1,\frac{dpr_{2}^{'}}{d+pr_{2}^{'}}}(I\times\mathbb{Z}^{d})}
			\end{equation*}
			\begin{equation*}
				+\|v_{2}\|_{L_{t}^{q_{1}}\ell_{x}^{r_{1}}(I\times \mathbb{Z}^{d})}+\Big\|v_{2}\Big\|_{L_{t}^{pq_{2}^{'}}\ell_{x}^{\frac{dpr_{2}^{'}}{d+pr_{2}^{'}}}(I\times \mathbb{Z}^{d})}.
			\end{equation*}
			In the order $(P_{n})\Rightarrow \widetilde{(P_{n})}\Rightarrow(P_{n+1})$, we can use the induction to prove the two propositions below.
			\begin{equation*}
				(P_{n}): \quad \|\widetilde{u}_{n}\|_{S_{0}(\left(-\infty,-T\right]\times\mathbb{Z}^{d})}\le 10\epsilon,
			\end{equation*}
			\begin{equation*}
				\widetilde{(P_{n})}: \quad \|N(\widetilde{u}_{n})\|_{L_{t}^{q_{2}'}\ell_{x}^{r_{2}'}(\left(-\infty,-T\right]\times \mathbb{Z}^{d})}=\||u_{n}|^{p-1}u_{n}\|_{L_{t}^{q_{2}'}\ell_{x}^{r_{2}'}(\left(-\infty,-T\right]\times \mathbb{Z}^{d})}\le \frac{\epsilon}{10C_{0}}.
			\end{equation*}
			Applying the H\"{o}lder inequality and the Sobolev embedding, we can also show that
			\begin{equation*}
				\|N(\widetilde{u}_{n+1})-N(\widetilde{u}_{n})\|_{N((-\infty,-T]\times \mathbb{Z}^{d})}\le\|N(\widetilde{u}_{n+1})-N(\widetilde{u}_{n})\|_{L_{t}^{q_{3}^{'}}\ell_{x}^{r_{3}^{'}}((-\infty,-T]\times \mathbb{Z}^{d})}\lesssim  \|u_{n+1}-u_{n}\|_{L_{t}^{q_{5}}\ell_{x}^{r_{5}}(-\infty,-T]\times \mathbb{Z}^{d})}
			\end{equation*}
			\begin{equation*}
				\times \left(\|u_{n+1}\|_{L_{t}^{q_{1}}\ell_{x}^{r_{1}}((-\infty,-T]\times \mathbb{Z}^{d})}^{p-1}+\|u_{n}\|_{L_{t}^{q_{1}}\ell_{x}^{r_{1}}((-\infty,-T]\times \mathbb{Z}^{d})}^{p-1}\right)\lesssim \epsilon^{p}\|\widetilde{u}_{n+1}-\widetilde{u}_{n}\|_{S((-\infty,-T]\times \mathbb{Z}^{d}))}.
			\end{equation*}
			Taking $\epsilon\ll \frac{1}{C_{0}}$, we can ensure that
			\begin{equation*}
				\|N(\widetilde{u}_{n+1})-N(\widetilde{u}_{n})\|_{N((-\infty,-T]\times \mathbb{Z}^{d})}\le \frac{1}{2C_{0}}\|\widetilde{u}_{n+1}-\widetilde{u}_{n}\|_{S((-\infty,-T]\times \mathbb{Z}^{d}))}.
			\end{equation*}
			Applying the Strichartz estimate again, we finally get
			\begin{equation*}
				\|\widetilde{u}_{n+1}-\widetilde{u}_{n}\|_{S(\left(-\infty,-T\right]\times \mathbb{Z}^{d})}\le \|DN(\widetilde{u}_{n})-DN(\widetilde{u}_{n-1})\|_{S(\left(-\infty,-T\right]\times \mathbb{Z}^{d})}
			\end{equation*}
			\begin{equation*}
				\le C_{0}\|N(\widetilde{u}_{n})-N(\widetilde{u}_{n-1})\|_{N((-\infty,-T]\times \mathbb{Z}^{d})}\le \frac{1}{2}\|\widetilde{u}_{n}-\widetilde{u}_{n-1}\|_{S(\left(-\infty,-T\right]\times \mathbb{Z}^{d})}.
			\end{equation*}
			Thus, $\lbrace\widetilde{u}_{n}\rbrace$ is a Cauchy sequence in $S(\left(-\infty,-T\right]\times \mathbb{Z}^{d})$, with limit $\widetilde{u}$ as the solution of equation ($\ast_{-\infty},F_{-}$).
		\end{proof}
		\vspace{10pt}
		\subsection{Proof of Theorem \ref{S2}, \ref{S3}}
		From the following Duhamel formula of DNLW (\ref{DNLW})
		\begin{equation*}
			\widetilde{u}(t)=U_{0}(t)F_{0}+\mu\int_{0}^{t}U_{0}(t-s)N(\widetilde{u}(s))ds, \quad F_{0}=(f,g),
		\end{equation*}
		the asymptotic completeness is equivalent to the conditional convergence (in $X$) of 
		\begin{equation}\label{convergence}
			\int_{0}^{\pm\infty}U_{0}(-s)N(\widetilde{u}(s))ds.
		\end{equation}

		 In the condition of Theorem \ref{S1} or Theorem \ref{S2}, we can find $0<q,r<\infty$, such that $(q,r), (pq^{'},p\frac{dr^{'}}{d+r^{'}})$ are all $\beta_{d}$-admissible. Then using the Strichartz estimate, we can get
		\begin{equation*}
			\left\|\int_{t_{1}}^{t_{2}}U_{0}(-s)N(\widetilde{u}(s))ds\right\|_{X}\le \left\|\int_{t_{1}}^{t_{2}}U_{0}(-s)N(\widetilde{u}(s))ds\right\|_{S^{0}(\mathbb{R}\times\mathbb{Z}^{d})}
		\end{equation*}
		\begin{equation*}
			\lesssim \Big\|N(\widetilde{u})\Big\|_{L_{t}^{q^{'}}\ell_{x}^{\frac{dr^{'}}{d+r^{'}}}([t_{1},t_{2}]\times\mathbb{Z}^{d})}=\Big\|u\Big\|_{L_{t}^{pq^{'}}\ell_{x}^{p\frac{dr^{'}}{d+r^{'}}}([t_{1},t_{2}]\times\mathbb{Z}^{d})}\le \|\widetilde{u}\|_{S^{0}([t_{1},t_{2}]\times\mathbb{Z}^{d})}.
		\end{equation*}
		Similarly, in the condition of Theorem \ref{S1.5} or Theorem \ref{S3}, we can also take $0<q,r<\infty$, such that $(q,r), (pq^{'},\frac{dpr^{'}}{d+pr^{'}})$ are all $\beta_{d}$-admissible. Thus, we deduce that
		\begin{equation*}
			\left\|\int_{t_{1}}^{t_{2}}U_{0}(-s)N(\widetilde{u}(s))ds\right\|_{X}\lesssim \|\widetilde{u}\|_{S([t_{1},t_{2}]\times\mathbb{Z}^{d})}.
		\end{equation*}
		In conclusion, the conditional convergence of the integral (\ref{convergence}) in Theorem \ref{S2} and Theorem \ref{S3}, can be reduced to $\widetilde{u}\in S^{0}(\mathbb{R}\times\mathbb{Z}^{d})$ and $\widetilde{u}\in S(\mathbb{R}\times\mathbb{Z}^{d})$, respectively.
		
		Then we just need the following simple lemma to finish the proof of Theorem \ref{S2} and Theorem \ref{S3}.
		\begin{lemma}\label{3.19}
			With the assumptions of Theorem \ref{S2}, there exists $\delta>0$, such that if the initial data $F_{0}=(f,g)$ satisfies $\|f\|_{\ell^{2}(\mathbb{Z}^{d}}+\|g\|_{\ell^{\frac{2d}{d+2}}(\mathbb{Z}^{d})}<\delta$, then the equation (\ref{DNLW}) has a unique solution $\widetilde{u}\in S^{0}(\mathbb{R}\times\mathbb{Z}^{d})$, with
			\begin{equation*}
				\|\widetilde{u}\|_{S^{0}(\mathbb{R}\times\mathbb{Z}^{d})}\lesssim \delta.
			\end{equation*}
			Parallelly, with the assumptions of Theorem \ref{S3}, there exists $\delta>0$, such that if the initial data $F_{0}$ satisfies $\|F_{0}\|_{X}<\delta$, then the equation (\ref{DNLW}) has a unique solution $\widetilde{u}\in S(\mathbb{R}\times\mathbb{Z}^{d})$, with
			\begin{equation*}
				\|\widetilde{u}\|_{S(\mathbb{R}\times\mathbb{Z}^{d})}\lesssim \delta.
			\end{equation*}
		\end{lemma}
		\begin{proof}
			For convenience, we just prove the latter statement, and the proof of former one follows similarly.
			
			We still take $q_{i},r_{i}$, $i=1,2,3,4,5$, which are the parameters in the proof of Theorem \ref{S1.5}. Then we consider the following closed set
			\begin{equation*}
				\Gamma:=\left\lbrace F\in S(\mathbb{R}\times\mathbb{Z}^{d})\Big|\|F\|_{S(\mathbb{R}\times\mathbb{Z}^{d})}\le M\delta\right\rbrace,
			\end{equation*}
			where $M>0$ will be determined later. And we denote map $\Lambda$ as
			\begin{equation*}
				\Lambda: F\mapsto U_{0}(t)F_{0}+\mu\int_{0}^{t}U_{0}(t-s)N(F(s))ds, \quad F=(F_{1},F_{2}).
			\end{equation*}
			From the contraction mapping, it suffices to show $\Lambda$ is a contraction, mapping $\Gamma$ to $\Gamma$.
			
			For $F\in \Gamma$, we can get
			\begin{equation*}
				\|\Lambda(F)\|_{S(\mathbb{R}\times\mathbb{Z}^{d})}\le \|U_{0}(t)F_{0}\|_{S(\mathbb{R}\times\mathbb{Z}^{d})}+\left\|\int_{0}^{t}U_{0}(t-s)N(F(s))ds\right\|_{S(\mathbb{R}\times\mathbb{Z}^{d})}
			\end{equation*}
			\begin{equation*}
				\le C\delta+C\Big\|N(F)\Big\|_{L_{t}^{q_{2}^{'}}\ell_{x}^{\frac{dr_{2}^{'}}{d+r_{2}^{'}}}(\mathbb{R}\times\mathbb{Z}^{d})}=C\delta+C\Big\|F\Big\|_{L_{t}^{pq_{2}^{'}}\ell_{x}^{p\frac{dr_{2}^{'}}{d+r_{2}^{'}}}(\mathbb{R}\times\mathbb{Z}^{d})}^{p}\le C\delta+C(M\delta)^{p}.
			\end{equation*}
			We can take $C\ll M\ll \frac{1}{\delta}$ to ensure $\|\Lambda(F)\|_{S(\mathbb{R}\times\mathbb{Z}^{d})}\le M\delta$, which shows that $\Lambda$ maps $\Gamma$ to $\Gamma$.
			
			For $F,G\in S(\mathbb{R}\times\mathbb{Z}^{d})$, we can deduce that
			\begin{equation*}
				\|\Lambda(F)-\Lambda(G)\|_{S(\mathbb{R}\times\mathbb{Z}^{d})}=\left\|\int_{0}^{t}U_{0}(t-s)\left(N(F(s))-N(G(s))\right)ds\right\|_{S(\mathbb{R}\times\mathbb{Z}^{d})}
			\end{equation*}
			\begin{equation*}
				\le C\|N(F(s))-N(G(s))\|_{L_{t}^{q_{3}^{'}}\ell_{x}^{r_{3}^{'}}(\mathbb{R}\times\mathbb{Z}^{d})}
				\le C\|F-G\|_{L_{t}^{q_{5}}\ell_{x}^{r_{5}}(\mathbb{R}\times\mathbb{Z}^{d})}\left(\|F\|_{L_{t}^{q_{1}}\ell_{x}^{r_{1}}(\mathbb{R}\times \mathbb{Z}^{d})}^{p-1}+\|G\|_{L_{t}^{q_{1}}\ell_{x}^{r_{1}}(\mathbb{R}\times \mathbb{Z}^{d})}^{p-1}\right)
			\end{equation*}
			\begin{equation*}
				\le C(M\delta)^{p-1}\|F-G\|_{S(\mathbb{R}\times \mathbb{Z}^{d})}.
			\end{equation*}
			Taking $M\ll \frac{1}{\delta}$, we can make $\Lambda$ a contraction, with a unique fixed point $\widetilde{u}\in  \Gamma\subseteq S(\mathbb{R}\times \mathbb{Z}^{d})$ as the solution of the equation (\ref{DNLW}).
		\end{proof}
		\begin{proof}[Proof of Theorem \ref{S2}, \ref{S3}]
			It's a direct consequence of Lemma \ref{3.19} and the reduction in the beginning of this subsection.
		\end{proof}
		\begin{rem}\label{3.18}
			For the scattering theory of DNLW (\ref{DNLW}), especially the asymptotic completeness, we usually require some smallness conditions. In fact, such smallness conditions are necessary in some extent. We can use the method in \cite{27} to prove that the focusing DNLW (i.e. $\mu=1$) can blow up, even with some compactly supported initial data, which makes the scattering fail. 
			
			We will show that 
			\begin{equation*}
				F(t):=\sum_{x\in \mathbb{Z}^{d}}|u(x,t)|^{2},
			\end{equation*}
			will tend to infinity in finite time. For convenience, we now consider the real-valued initial data ($f,g$), then $u(x,t)$ is also real-valued, from the uniqueness.
			
			Therefore, it suffices to show
			\begin{equation*}
				(I): \; \left(F(t)^{-\alpha}\right)^{''}\le 0, \quad \forall t\ge 0; \quad  (II): \; \left(F(t)^{-\alpha}\right)^{'}<0, \quad t=0,
			\end{equation*}
			where $\alpha>0$ will be determined later.
			
			For (II), we have 
			\begin{equation*}
				\left(F(t)^{-\alpha}\right)^{'}\Big |_{t=0}=-\alpha F(0)^{-(1+\alpha)} F^{'}(0)=-2\alpha F(0)^{-(1+\alpha)}\sum_{x\in \mathbb{Z}^{d}}f(x)g(x).
			\end{equation*}
			Then we can take $fg\ge 0$, which ensures (II).
			
			Simple calculations show that
			\begin{equation}\label{3.12}
				F^{'}(t)=2\sum_{x\in\mathbb{Z}^{d}}u(x,t)\partial_{t}u(x,t),
			\end{equation}
			\begin{equation}\label{3.121}
				F^{''}(t)=2\sum_{x\in \mathbb{Z}^{d}}u(x,t)\partial_{t}^{2}u(x,t)+(\partial_{t}u(x,t))^{2}
			\end{equation}
			\begin{equation*}
				=4(\alpha+1)\sum_{x\in \mathbb{Z}^{d}}(\partial_{t}u(x,t))^{2}+2\sum_{x\in \mathbb{Z}^{d}}u(x,t)\partial_{t}^{2}u(x,t)-(2\alpha+1)(\partial_{t}u(x,t))^{2}.
			\end{equation*}
			As (I) is equivalent to 
			\begin{equation}\label{omega}
				F^{''}(t)F(t)-(\alpha+1)(F^{'}(t))^{2}\ge0, \quad \forall t\ge 0,
			\end{equation}
			we can substitute (\ref{3.12}), (\ref{3.121}) into (\ref{omega}), and derive that
			\begin{equation*}
				(\ref{omega})=4(\alpha+1)\left\lbrace\sum_{x\in \mathbb{Z}^{d}}u^{2}(x,t)\sum_{x\in \mathbb{Z}^{d}}(\partial_{t}u(x,t))^{2}-\sum_{x\in \mathbb{Z}^{d}}\left(u(x,t)\partial_{t}u(x,t)\right)^{2}\right\rbrace
			\end{equation*}
			\begin{equation*}
				+2F(t)\left\lbrace\sum_{x\in \mathbb{Z}^{d}}u(x,t)\partial_{t}^{2}u(x,t)-\sum_{x\in \mathbb{Z}^{d}}(2\alpha+1)(\partial_{t}u(x,t))^{2}\right\rbrace.
			\end{equation*}
			The first term is obviously nonnegative, from the Cauchy-Schwarz inequality. We denote the part in the second curly bracket as $H(t)$, and just need to ensure that $H(t)$ is also nonnegative.
			
			Notice that we have the energy conservation of DNLW as follows.
			\begin{equation*}
				E(t):=\frac{1}{2}\sum_{x\in \mathbb{Z}^{d}}\left|\nabla u(x,t)\right|^{2}+(\partial_{t}u(x,t))^{2}-\frac{2}{p+1}|u(x,t)|^{p+1}\equiv E(0),
			\end{equation*}
		      where $\nabla u:=\mathcal{F}^{-1}(\omega(\xi)\mathcal{F}(u))$.
			
			Then we can take $\alpha=\frac{p-1}{4}>0$, and get 
			\begin{equation*}
				H(t)=\sum_{x\in \mathbb{Z}^{d}}|u(x,t)|^{p+1}-\sum_{x\in \mathbb{Z}^{d}}|\nabla u(x,t)|^{2}-(2\alpha+1)\sum_{x\in \mathbb{Z}^{d}}(\partial_{t}u(x,t))^{2}
			\end{equation*}
			\begin{equation*}
				=-(p+1)E(0)+\frac{p-1}{2}\sum_{x\in \mathbb{Z}^{d}}|\nabla u(x,t)|^{2}.   
			\end{equation*}
			We now only need to take ($f,g$), such that $E(0)\le0$, which is very trivial. Thus we have showed the solution of the focusing DNLW will blow up with some compactly supported large data, which directly breaks the scattering theory.
		\end{rem}
		\newpage
		\appendix
		\section{}
		In this Appendix A, we will introduce some basic concepts of the uniform estimate of the oscillatory integrals, following the notations in \cite{9} or \cite{2}.
		\begin{defi}
			For $r,s>0$, the space $\mathcal{H}_{r}(s)$ is defined as follows.
			\begin{equation*}
				\mathcal{H}_{r}(s):=\left\lbrace P\Big| P\in \mathcal{O}(B_{\mathbb{C}^{d}}(0,r))\cap C(\overline{B}_{\mathbb{C}^{d}}(0,r)), |P(w)|<s, \forall w\in \overline{B}_{\mathbb{C}^{d}}(0,r) \right\rbrace.
			\end{equation*}
		\end{defi}
		\begin{defi}
			Suppose that $h:\mathbb{R}^{d}\to \mathbb{R}$ is real analytic at $0$. We write 
			\begin{equation*}
				M(h)\curlyeqprec (\beta,p), \; \beta\le 0, p\in \mathbb{N},
			\end{equation*}
			if for $r>0$ sufficiently small, there exist $s>0$, $C>0$ and a neighborhood $\Omega\subseteq B_{\mathbb{R}^{d}}(0,r)$ of the origin, such that
			\begin{equation*}
				\left|J_{h+P,\zeta}(\tau)\right|\le C(1+|\tau|)^{\beta}\log^{p}(2+|\tau|)\|\zeta\|_{C^{N}(\Omega)}, \; \forall \tau\in \mathbb{R}, \zeta\in C_{c}^{\infty}(\Omega), P\in \mathcal{H}_{r}(s),
			\end{equation*}
			where $N=N(h)\in \mathbb{N}$, with 
			\begin{equation*}
				J_{S,\zeta}(\tau):=\int_{\mathbb{R}^d} e^{i\tau S(x)}\zeta(x)dx,
			\end{equation*}
			\begin{equation*}
				\|\zeta\|_{C^{N}(\Omega)}:=\sup\left\lbrace|\partial^{\gamma}\zeta(\xi)|\Big|\xi\in\Omega, \gamma\in \mathbb{N}^{d}, |\gamma|\le N\right\rbrace.
			\end{equation*}
			We have the following writing convention.
			\begin{itemize}
				\item We write $M(h,\xi)\curlyeqprec(\beta,p)$, if 
				\begin{equation*}
					M(\tau_{\xi}h)\curlyeqprec(\beta,p), \; where \; \tau_{\xi}h(y)=h(y+\xi), \forall y\in\mathbb{R}^{d}.
				\end{equation*}
				\item We write $M(h_{2})\curlyeqprec M(h_{1})+(\beta_{2},p_{2})$, if 
				\begin{equation*}
					M(h_{1})\curlyeqprec(\beta_{1},p_{1}) \;\; implies \; \; M(h_{2})\curlyeqprec(\beta_{1}+\beta_{2},p_{1}+p_{2}).
				\end{equation*}
			\end{itemize} 
			
		\end{defi}
		\vspace{10pt}
		Let $\gamma=(\gamma_{1},\cdots, \gamma_{d})\in \mathbb{R}^{d}$, with $\gamma_{i}>0, \forall i=1,\cdots,d$. For any $c>0$, we define the dilation as follows.
		\begin{equation*}
			\delta_{c}^{\gamma}(\xi):=\left(c^{\gamma_{1}}\xi_{1},\cdots,c^{\gamma_{d}}\xi_{d}\right), \forall \xi\in \mathbb{R}^{d}.
		\end{equation*}
		\begin{defi}
			A polynomial $f$ on $\mathbb{R}^{d}$ is called $\gamma$-homogeneous of degree $\rho$, if 
			\begin{equation*}
				f\circ\delta_{c}^{\gamma}(\xi)=c^{\rho}f(\xi), \forall \xi\in\mathbb{R}^{d}, c>0.
			\end{equation*}
		\end{defi}
		Let $\mathcal{E}_{\gamma,d}$ be the set of all $\gamma$-homogeneous polynomials on $\mathbb{R}^{d}$ of degree $1$. $H_{\gamma,d}$ is the set of all functions that are real-analytic at $0$, with the Taylor series having the form of $\sum_{\gamma\cdot\alpha>1}a_{\alpha}\xi^{\alpha}$, i.e. the monomial is $\gamma$-homogeneous of degree $>1$.
		
		Then we briefly introduce some useful lemmas, which can be seen in \cite{9} or \cite{2}.
		\begin{lemma}\label{j}
			If $h$ is real analytic at $0$ and $\nabla h(0)\ne 0$, then
			\begin{equation*}
				M(h)\curlyeqprec (-n,0), \forall n\in\mathbb{N}.		  
			\end{equation*}
		\end{lemma}
		\begin{lemma}\label{we}
			If $h\in \mathcal{E}_{\gamma,d}$ and $P\in H_{\gamma,d}$, then 
			\begin{equation*}
				M(h+P)\curlyeqprec M(h).
			\end{equation*}
		\end{lemma}
		\begin{lemma}\label{Q}
			Let $m,n\ge 1$ and 
			\begin{equation*}
				h_{2}(\xi,y)=h_{1}(\xi)+Q(y), \forall \xi\in\mathbb{R}^{n}, y\in \mathbb{R}^{m},
			\end{equation*}
			where $Q(y)=\sum_{j=1}^{m}c_{j}y_{j}^{2}$, with $c_{j}=\pm1, j=1,\cdots,m$. Then we have
			\begin{equation*}
				M(h_{2})\curlyeqprec M(h_{1})+(-\frac{m}{2},0).
			\end{equation*}
		\end{lemma}
		\newpage
		\section{}
		Next, we introduce the Newton polyhedra and related theorems. The concepts can be referred to \cite{6,10}
		
		Let $S:\mathbb{R}^{d}\to \mathbb{R}$ be real-analytic at $0$, satisfying
		\begin{equation}\label{a}
			S(0)=0,\quad \nabla S(0)=0.
		\end{equation} 
		Suppose the corresponding Taylor series at $0$ is 
		\begin{equation*}
			S(\xi)=\sum_{\gamma\in\mathbb{N}^{d}}s_{\gamma}\xi^{\gamma}.
		\end{equation*}
		We define the Taylor support set $\supp S:=\left\lbrace \gamma\in\mathbb{N}^{d}\Big|s_{\gamma}\ne 0 \right\rbrace$.
		\begin{defi}
			The Newton polyhedron $\mathcal{N}(S)$ of such $S$, is the convex hull of 
			\begin{equation*}
				\bigcup_{\gamma\in \supp(S)}\left(\gamma+\mathbb{R}_{+}^{d}\right), \quad  \mathbb{R}_{+}^{d}:=\left\lbrace \xi\in\mathbb{R}^{d}\Big|\xi_{i}>0, i=1,\cdots,d\right\rbrace.
			\end{equation*}
			If $\mathcal{P}$ is a face of the Newton polyhedron $\mathcal{N}(S)$, then we denote $S_{\mathcal{P}}(\xi):=\sum_{\gamma\in \mathcal{P}}s_{\gamma}\xi^{\gamma}$.
		\end{defi}
		\begin{defi}
			S is called $\mathbb{R}-$nondegenerate if for any compact face $\mathcal{P}$,
			\begin{equation*}
				\bigcap_{i=1}^{d}\left\lbrace\xi\in\mathbb{R}^{d}\Big|\partial_{i}S_{\mathcal{P}}(\xi)=0\right\rbrace \subseteq\bigcup_{i=1}^{d}\left\lbrace\xi\in\mathbb{R}^{d}\Big|\xi_{i}=0\right\rbrace ,
			\end{equation*}
			i.e. $\nabla S_{\mathcal{P}}$ is non-vanishing on $(\mathbb{R}-\lbrace0\rbrace)^{d}$.
		\end{defi}
		\begin{defi}
			If $\supp(S)\ne \emptyset$, then the Newton distance $d_{S}$ is defined as
			\begin{equation*}
				d_{S}:=\inf\left\lbrace d>0\Big| (d,\cdots,d)\in \mathcal{N}(S)\right\rbrace.
			\end{equation*}
			The principal face $\pi_{S}$ is the face of minimal dimension that intersects with $\lbrace \xi_{1}=\cdots=\xi_{d}\rbrace$.
			We also denote $S_{\pi}:=S_{\pi_{S}}$, $k_{S}:=d-\dim_{\mathbb{R}^{d}}(\pi_{S})$, where $\dim_{\mathbb{R}^{d}}(\cdot)$ means affine dimension. 
		\end{defi}
		\begin{defi}
			The height of $S$ is defined as follows.
			\begin{equation*}
				h_{S}:=\sup\lbrace d_{S,\xi}\rbrace,
			\end{equation*}
			where the supremum is taken over all local analytic coordinate system $\xi$, preserving the $0$, and $d_{S,\xi}$ is the corresponding Newton distance. 
			
			A coordinate system $\xi_{\ast}$ is called adapted, if $d_{S,\xi^{\ast}}=h_{s}$.
		\end{defi}
		\vspace{7.5pt}
		To recognize if a coordinate system is adapted, there is a very useful theorem for $d=2$.
		\begin{thm}\label{superadapted}
			Let $d=2$, if the principal face $\pi_{S}$ is compact and both the functions $S_{\pi}(1,y)$ and $S_{\pi}(-1,y)$ have no real root of order $\ge d_{S}$ other than possibly $y=0$, then the coordinate system is adapted. 
		\end{thm}
		\begin{proof}
			See e.g. \cite{6} or \cite{7}.
		\end{proof}
		It is well-known that the oscillatory integral $J_{S,\zeta}(\tau)$ has the following asymptotic expansion.
		\begin{equation*}
			J_{S,\zeta}(\tau)\approx \sum_{\beta}\sum_{\rho=1}^{d-1}c_{\beta,\rho,\zeta}\tau^{\beta}\log^{\rho}(\tau),
		\end{equation*}
		where $\beta$ runs through finitely many arithmetic progressions of negative rational numbers. More details about this asymptotic expansion can be found in \cite{10,11}. Let $(\beta_{S},\rho_{S})$ be the maximum over all pair $(\beta,\rho)$, such that for any neighborhood $\Omega$ of the $0$, there exists $\zeta\in C_{c}^{\infty}(\Omega)$ and the corresponding $c_{\beta_{S},\rho_{S},\zeta}\ne 0$. And we usually call such $\beta_{S}$ as the oscillatory index of $S$ and $\rho_{S}$ as the multiplicity of $S$.
		
		Then we have the following theorem for $d=2$, that connects the adapted coordinate system and $(\beta_{S},\rho_{S})$.
		\begin{thm}\label{uniform estimate}
			If $S$ satisfies (\ref{a}), then there exists a coordinate system that is adapted to $S$. Furthermore, we have 
			\begin{equation*}
				M(S)\curlyeqprec(\beta_{S},\rho_{S}), \; \beta_{S}=-\frac{1}{h_{S}}.
			\end{equation*}
		\end{thm}
		\begin{proof}
			See e.g. \cite{21}.
		\end{proof}
		Now we can use the Newton polyhedron and the above theorems to derive the uniform estimate of some specific oscillatory integrals, which has been used in Theorem \ref{uniform}.
		\begin{lemma}\label{-2/3}
			For $k\in \mathbb{N}$, $\xi=(\xi_{1},\cdots,\xi_{d})$, we have
			\begin{equation*}
				M(\xi_{1}^{k})\curlyeqprec(-k,0), \quad M(\xi_{1}^{2}\xi_{2}\pm\xi_{1}\xi_{2}^{2})\curlyeqprec(-\frac{2}{3},0),
			\end{equation*}
			\begin{equation*}
				M(\xi_{1}^{2}\xi_{2}\pm\xi_{2}^{2})\curlyeqprec(-\frac{3}{4},0), \quad M(\xi_{1}^{2}\xi_{2}\pm\xi_{2}^{2}\pm\xi_{1}^{4})\curlyeqprec(-\frac{3}{4},0).
			\end{equation*}
		\end{lemma}
		\begin{proof}
			Using the Van der Corput lemma (see e.g. \cite{5}), we can prove the first statement. For the second statement, we have
			\begin{equation*}
				\mathcal{N}(S)=\mathcal{N}(\xi_{1}^{2}\xi_{2}\pm\xi_{1}\xi_{2}^{2})=\left\lbrace(x,y)\in \mathbb{R}_{+}^{2}\Big|x+y>3, x>1, y>1\right\rbrace,
			\end{equation*}
			\begin{equation*}
				\pi=\left\lbrace(x,y)\in \mathbb{R}_{+}^{2}\Big|x+y=3, x\in[1,2]\right\rbrace,\; S_{\pi}=x^{2}y\pm xy^{2},\; d_{S}=\dfrac{3}{2}.
			\end{equation*}
			For the third statement, we have
			\begin{equation*}
				\mathcal{N}(S)=\mathcal{N}(\xi_{1}^{2}\xi_{2}\pm\xi_{2}^{2})=\left\lbrace(x,y)\in \mathbb{R}_{+}^{2}\Big|x+2y>4, y>1\right\rbrace,
			\end{equation*}
			\begin{equation*}
				\pi=\left\lbrace(x,y)\in \mathbb{R}_{+}^{2}\Big|x+2y=4, x\in[0,2]\right\rbrace,\; S_{\pi}=x^{2}y\pm y^{2},\; d_{S}=\dfrac{4}{3}.
			\end{equation*}
			For the fourth statement, we have 
			\begin{equation*}
				\mathcal{N}(S)=\mathcal{N}(\xi_{1}^{2}\xi_{2}\pm\xi_{2}^{2}\pm \xi_{1}^{4})=\left\lbrace(x,y)\in \mathbb{R}_{+}^{2}\Big|x+2y>4\right\rbrace,
			\end{equation*}
			\begin{equation*}
				\pi=\left\lbrace(x,y)\in \mathbb{R}_{+}^{2}\Big|x+2y=4, x\in[0,4]\right\rbrace,\; S_{\pi}=x^{2}y\pm y^{2}\pm x^{4},\; d_{S}=\dfrac{4}{3}.
			\end{equation*}
			Based on Theorem \ref{superadapted} and Theorem \ref{uniform estimate}, we derive the  statements.
			
		\end{proof}

		\newpage
		\section*{Acknowledgement}
		The author is supported by NSFC, No. 123B1035 and is grateful to Prof. B. Hua for helpful suggestions.
		
		\section*{Conflict of interest statement}
		The author does not have any possible conflict of interest.
		
		\section*{Data availability statement}
		The manuscript has no associated data.
		\bigskip
		\bigskip

		\bibliographystyle{alpha}
		\bibliography{reference}
		
	\end{document}